\documentclass[12pt]{amsart}

\usepackage{amsfonts,amssymb,enumerate,color,bm,hhline,makecell,array}
\usepackage{array}
\usepackage{mathrsfs}
\usepackage{comment}
\usepackage[all,ps,cmtip]{xy} 
\usepackage[normalem]{ulem}
\usepackage[colorlinks=true,citecolor=blue, urlcolor=blue, linkcolor=blue, pagebackref]{hyperref}
\usepackage{tikz-cd}
\usepackage{amscd}
\usepackage[shortlabels]{enumitem}
\usepackage{tensor}
\usepackage{tikz}
\usetikzlibrary{matrix,arrows}
\usepackage{extarrows}

\newcommand{\CC}{{\mathbb C}}
\newcommand{\FF}{{\mathbb F}}
\newcommand{\NN}{{\mathbb N}}
\newcommand{\PP}{{\mathbb P}}
\newcommand{\QQ}{{\mathbb Q}}

\newcommand{\ZZ}{{\mathbb Z}}

\def\C{\mathbb{C}}
\def\D{\mathbb{D}}

\def\P{\mathbb{P}}

\def\R{\mathbb{R}}
\def\Z{\mathbb{Z}}

\let\mathcal\mathscr

\def\cD{\mathcal{D}}
\def\cE{\mathcal{E}}
\def\cF{\mathcal{F}}
\def\cG{\mathcal{G}}
\def\cH{\mathcal{H}}

\def\cL{\mathcal{L}}

\def\cO{\mathcal{O}}

\def\cU{\mathcal{U}}

\def\cX{\mathcal{X}}
\def\cY{\mathcal{Y}}

\def\phi{{\varphi}}

\DeclareMathOperator{\Aut}{Aut}

\DeclareMathOperator{\Bl}{Bl}

\DeclareMathOperator{\ch}{ch}
\DeclareMathOperator{\CH}{CH}
\DeclareMathOperator{\cl}{cl}

\DeclareMathOperator{\coh}{coh}

\DeclareMathOperator{\codim}{codim}

\DeclareMathOperator{\Def}{Def}

\DeclareMathOperator{\Db}{D\textsuperscript{\rm b}}

\def\div{\mathop{\rm div}\nolimits}
\DeclareMathOperator{\divisore}{div}

\DeclareMathOperator{\Ext}{Ext}
\DeclareMathOperator{\Fix}{Fix}

\DeclareMathOperator{\Gr}{Gr}

\DeclareMathOperator{\Hom}{Hom}
\DeclareMathOperator{\Id}{Id}

\def\Im{\mathop{\rm Im}\nolimits}

\DeclareMathOperator{\jb}{\ov{J}}

\DeclareMathOperator{\Mov}{Mov}
\DeclareMathOperator{\mult}{mult}
\DeclareMathOperator{\Nef}{Nef}
\DeclareMathOperator{\NS}{NS}

\DeclareMathOperator{\Pic}{Pic}
\DeclareMathOperator{\Pos}{Pos}

\DeclareMathOperator{\Sp}{Sp}

\DeclareMathOperator{\suppdet}{Supp_{det}}

\DeclareMathOperator{\Stab}{Stab}

\DeclareMathOperator{\supp}{supp}
\DeclareMathOperator{\td}{td}

\newcommand{\la}{\langle}
\newcommand{\ov}{\overline}
\newcommand{\ra}{\rangle}

\newcommand{\wt}{\widetilde}

\def\lra{\longrightarrow}
\def\llra{\hbox to 10mm{\rightarrowfill}}
\def\lllra{\hbox to 15mm{\rightarrowfill}}

\def\llla{\hbox to 10mm{\leftarrowfill}}
\def\lllla{\hbox to 15mm{\leftarrowfill}}

\def\dra{\dashrightarrow}

\def\thra{\twoheadrightarrow}
\def\hra{\hookrightarrow}

\newtheorem{lemm}{Lemma}[section]
\newtheorem{theo}[lemm]{Theorem}
\newtheorem*{MainThm*}{Main Theorem}
\newtheorem{coro}[lemm]{Corollary}
\newtheorem{prop}[lemm]{Proposition}

\theoremstyle{definition}
\newtheorem{defi}[lemm]{Definition}
\newtheorem{rema}[lemm]{Remark}

\newtheorem{exam}[lemm]{Example}

\theoremstyle{remark}
\newtheorem*{remark*}{Remark}
\newtheorem*{note*}{Note}
\newtheorem*{lemmared*}{Lemma}

\def\blank{\underline{\hphantom{A}}}

\makeatletter
\@namedef{subjclassname@2020}{%
  \textup{2020} Mathematics Subject Classification}
\makeatother

\makeatletter
\def\@tocline#1#2#3#4#5#6#7{\relax
  \ifnum #1>\c@tocdepth 
  \else
    \par \addpenalty\@secpenalty\addvspace{#2}%
    \begingroup \hyphenpenalty\@M
    \@ifempty{#4}{%
      \@tempdima\csname r@tocindent\number#1\endcsname\relax
    }{%
      \@tempdima#4\relax
    }%
    \parindent\z@ \leftskip#3\relax \advance\leftskip\@tempdima\relax
    \rightskip\@pnumwidth plus4em \parfillskip-\@pnumwidth
    #5\leavevmode\hskip-\@tempdima
      \ifcase #1
       \or\or \hskip 1em \or \hskip 2em \else \hskip 3em \fi%
      #6\nobreak\relax
    \dotfill\hbox to\@pnumwidth{\@tocpagenum{#7}}\par
    \nobreak
    \endgroup
  \fi}
\makeatother

\allowdisplaybreaks

\usepackage{enumitem}
\setlist[itemize]{noitemsep,nolistsep}
\setlist[enumerate]{noitemsep,nolistsep}
\usepackage{colortbl}
\usepackage{enumerate}

\title{The geometry of antisymplectic involutions, I}

\begin{document}

\author{Laure Flapan}
\address{\parbox{0.9\textwidth}{Michigan State University\\[1pt]
Department of Mathematics\\[1pt]
619 Red Cedar Road, East Lansing, MI 48824, USA
\vspace{1mm}}}
\email{{flapanla@msu.edu}}

\author{Emanuele Macr\`i}
\address{\parbox{0.9\textwidth}{Universit\'e Paris-Saclay\\[1pt]
CNRS, Laboratoire de Math\'ematiques d'Orsay\\[1pt]
Rue Michel Magat, B\^at. 307, 91405 Orsay, France
\vspace{1mm}}}
\email{{emanuele.macri@universite-paris-saclay.fr}}

\author{Kieran G.~O'Grady}
\address{\parbox{0.9\textwidth}{Sapienza Universit\`a di Roma\\[1pt]
Dipartimento di Matematica\\[1pt]
P.le A. Moro 5, 00185 Roma, Italia
\vspace{1mm}}}
\email{{ogrady@mat.uniroma1.it}}

\author{Giulia Sacc\`a}
\address{\parbox{0.9\textwidth}{Columbia University\\[1pt]
Department of Mathematics\\[1pt]
2990 Broadway, New York, NY 10027, USA
\vspace{1mm}}}
\email{{gs3032@columbia.edu}}

\subjclass[2020]{14C20, 14D06, 14D20, 14F08, 14J42, 14J60}
\keywords{Projective hyperk\"ahler manifolds, antisymplectic involutions, Lagrangian fibrations, moduli spaces, Bridgeland stability}
\thanks{L.F.~was partially supported by the NSF grants DMS-1803082, DMS-1645877, as well as DMS-1440140, while in residence at the Mathematical Sciences Research Institute (MSRI) in Berkeley, California during the Spring 2019 semester. E.M.~was partially supported by the NSF grant DMS-1700751, by the Institut des Hautes \'Etudes Scientifiques (IH\'ES), by a Poste Rouge CNRS at Universit\'e Paris-Sud and by the ERC Synergy Grant ERC-2020-SyG-854361-HyperK. K.O'G.~was partially supported by PRIN 2017 \lq\lq Moduli and Lie Theory\rq\rq.  G.S.~was partially supported by the NSF grant DMS-1801818.}

\dedicatory{\`A Olivier Debarre, avec admiration et amiti\'e}

\begin{abstract}
We study fixed loci of antisymplectic involutions on projective hyperk\"ahler manifolds of $\mathrm{K3}^{[n]}$-type.
When the involution is induced by an ample class of square 2 in the Beauville-Bogomolov-Fujiki lattice, we show that the number of connected components of the fixed locus is equal to the divisibility of the class, which is either 1 or 2.
\end{abstract}

\maketitle

\tableofcontents
\setcounter{tocdepth}{1}

\section{Introduction}\label{sec:intro}

An involution of a compact irreducible hyperk\"ahler (HK) manifold is antisymplectic if it acts as $(-1)$ on the space of global holomorphic $2$-forms.
The goal of this paper is to describe fixed loci of antisymplectic involutions of HK manifolds of $\mathrm{K3}^{[n]}$-type, i.e., deformations of the Hilbert scheme of length $n$ subschemes of a K3 surface.
The involutions that we study have invariant sublattice in the second integral cohomology spanned by a class of square~2.
HK's with such an involution vary in families of the maximum allowable dimension and appear quite frequently; they are the simplest antisymplectic involutions that deform in codimension one.
We carry out our analysis by letting the HK degenerate.
As a matter of fact, the degeneration is quite mild --- it is the contraction of a (smooth) HK which is birational to a Lagrangian HK---but the involution on the singular variety is descended from an  antisymplectic involution on the smooth HK whose $(+1)$ eigenspace in $H^2$ has rank $2$.
 Wall crossing techniques allow us to reduce everything to an analysis of fiberwise  involutions of Lagrangian HK's.

The motivation for studying these fixed loci comes from two different directions. 
The first direction is to explore further the relationship between HK manifolds of $\mathrm{K3}^{[n]}$-type and Fano manifolds of K3 type, seen in classical constructions \cite{BD:cubic,DV:HK,LLSvS:cubics,Kuz:Kuchle} as well as recent works \cite{IM:FanoCY,FM:FanoK3,BFM:DV}. These constructions produce a HK manifold of $\mathrm{K3}^{[n]}$-type  starting from a Fano manifold  of K3 type. The results of this paper together with the subsequent paper yield a reverse process in which one starts with a HK manifold $X$ of $\mathrm{K3}^{[n]}$-type and produces a corresponding Fano manifold arising as a connected component of the fixed locus of an antisymplectic involution on $X$.

The second motivating direction for studying these fixed loci is to produce covering families of Lagrangian cycles on HK manifolds. Low-dimensional examples, discussed below, suggest that one component of the fixed locus of an antisymplectic involution of a HK of $\mathrm{K3}^{[n]}$-type with the maximum number of moduli  might, up to a multiple, be contained in a covering family of Lagrangian cycles.

\subsection{Motivating examples}\label{subsec:background}

The following are examples of antisymplectic involutions of HK's of $\mathrm{K3}^{[n]}$-type with the maximum number of moduli:
\begin{enumerate}[(a)]
\item\label{enum:background1} Let $(S,h)$ be a polarized K3 surface of degree $2$, and let $\tau\in\Aut(S)$ be the covering involution of the $2:1$ map $S\to |h|^{\vee}\cong\PP^2$.
\item\label{enum:background2} Let $f\colon X\to Y$ be the natural double cover of a general EPW sextic $Y\subset \PP^5$, and let $\tau\in\Aut(X)$  be the covering involution of $f$;  $X$ is of $\mathrm{K3}^{[2]}$-type (see~\cite[Theorem 1.1]{Kieran:doubleEPW}).
\item\label{enum:background3} Let $Y\subset \PP^5$ be a smooth cubic fourfold not containing a plane, and let $Z$ be the associated Lehn--Lehn--Sorger--van Straten (LLSvS) HK manifold of $\mathrm{K3}^{[4]}$-type \cite{LLSvS:cubics}. Recall that $Z$ is a specific rationally connected quotient of $M_3(Y)$, the irreducible component of the Hilbert scheme of $Y$ containing smooth twisted cubic rational curves. More precisely there is an open dense subset $Z^0\hookrightarrow Z$ whose points parametrize $\PP^2$-families of (possibly degenerate) twisted cubic curves on a cubic surface in $Y$. The complement $Z\setminus Z^0$ is isomorphic to $Y$ itself, and the point $y\in Y$ parametrizes the family of cubic curves  with an embedded point at $y$. One defines an involution  $Z^0\to Z^0$ by mapping a twisted cubic curve $C\subset Y$ to the equivalence class representing the residual intersection   of $Y$ and a quadric surface in $\la C\ra$ containing $C$. The involution extends to a regular involution $\tau$ and $Z\setminus Z_0\cong Y$ is a connected component of the fixed locus (see~\cite{Lehn:Oberwolfach,LLSvS:cubics}, \cite[Remark 2.19]{CCL:ChowLLSvS}, and \cite{LPZ:cubics}). 
\end{enumerate}

Let us look at the fixed loci in the above three examples. Recall that the fixed locus of an antisymplectic involution is a Lagrangian submanifold (see e.g., \cite[Lemma 1]{beauville:involutions}).
In~\ref{enum:background1} the fixed locus is an irreducible curve of general type and it moves in a family of Lagrangians covering $S$, namely the elements of 
$|3h|$.
In~\ref{enum:background2} the fixed locus is an irreducible surface of general type \cite[Corollary 1.12]{Ferretti:thesis},\cite[Proposition 5.6]{IM:doubleEPW}.
The fixed locus itself does not move, but there is a covering family of Lagrangian cycles containing $2\Fix(\tau)$.
In~\ref{enum:background3} the picture is even more intriguing.
The fixed locus $\Fix(\tau)$ has two connected components, one whose points are in one-to-one correspondence with Cayley cubic surfaces in $Y$, and one isomorphic to the variety $Y$ itself. As mentioned above we expect that a multiple of the first component moves in a  covering family of Lagrangian subvarieties. 

We predict that the above examples are instances of a general phenomenon.
In particular, the main result of this paper characterizes the number of connected components of the fixed locus of the antisymplectic involution considered. In the above three examples, the polarization has square $2$ with respect to the Beauville-Bogomolov-Fujiki (BBF) form and hence the involution $H^2(\tau)$ on $H^2(X,\Z)$ induced by $\tau$ is equal to the reflection in the polarization (see~\cite[Theorem 1.1]{Kieran:doubleEPW} for Case~\ref{enum:background2} and~\cite[Section 5.2]{LPZ:cubics} for Case~\ref{enum:background3}). In particular 
the families have maximal dimension. However there is 
 a key difference between the polarizations in Cases~\ref{enum:background1} and \ref{enum:background2} on one hand and Case~\ref{enum:background3} on the other. Recall that the divisibility of a non-zero element $v$ in a (non-degenerate) lattice $(L,(\blank,\blank))$ the positive generator of the ideal $(v, L)\subset\ZZ$.
In cases Cases~\ref{enum:background1} and \ref{enum:background2}  the polarization has divisibility $1$, while in Case~\ref{enum:background3} the polarization has divisibility  $2$.

\subsection{Statement of the main result}\label{subsec:MainResults}

First we recall some well-known results on HK's of $\mathrm{K3}^{[n]}$-type.
Let $X$ be such a HK, let $q_X$ be the BBF quadratic form on $H^2(X;\ZZ)$, and let $\lambda$ be an ample class on $X$ such that $q_X(\lambda)=2$.
By the Global Torelli Theorem (see~\cite{Verbitsky:torelli,Markman:Monodromy,Huybrechts:Torelli} and \cite[Theorem 1.3]{Markman:Survey}), there exists  an involution $\tau\in\Aut(X)$  such that the invariant lattice
$H^2(\tau)_{+}=\ZZ\lambda$, and this $\tau$  is unique (\cite{Bea:MukaiInvolution}; also \cite[Proposition 4.1]{Debarre:Survey}). Additionally, we recall that by the surjectivity of the period map and an explicit lattice computation, polarized $(X,\lambda)$ of $\mathrm{K3}^{[n]}$-type with $q_X(\lambda)=2$ and $\divisore(\lambda)=1$ exist for any $n$, while those with $\divisore(\lambda)=2$ exist if and only if $4\,|\,n$ (note that $\divisore(\lambda)$ divides 
$q_X(\lambda)$).

\begin{MainThm*}
Let $(X,\lambda)$ be a polarized HK of $K3^{[n]}$-type such that $q_X(\lambda)=2$, and let $\tau\in\Aut(X)$ be the involution associated to $\lambda$.
Then the number of connected components of  $\Fix(\tau)$ is equal to the divisibility $\divisore(\lambda)$.
\end{MainThm*}

In the case of divisibility~2, we show that the two connected components of the fixed locus can be distinguished by the behavior of any lift of the involution to the total space of the line bundle with class $\lambda$. More precisely (see Theorem ~\ref{thm:linearization}), over one component the action on the fibers of $\lambda$ is trivial while on the other component it is  multiplication by $(-1)$.

\subsection{Idea(s) from the proof of the Main Theorem}\label{subsec:pazzaidea}

Let $(n,d)\in\NN_{+}\times\{1,2\}$ with $d=1$ if $n\not\equiv 0\pmod{4}$. 
By Markman's results on monodromy (\cite{Markman:Survey}\cite[Proposition 3.2]{Apostolov:Moduli}), there is a single deformation class of polarized HK's $(X,\lambda)$ of $\mathrm{K3}^{[n]}$-type with $q_X(\lambda)=2$ and $\divisore(\lambda)=d$. 
It follows that the deformation class of $\Fix(\tau)$ is independent of $(X,\lambda)$, 
where $\tau\in\Aut(X)$ is the unique involution such that $H^2(\tau)_{+}=\ZZ\lambda$. Hence in order to prove our results it suffices to prove them for a single $(X,\lambda)$. This is more easily said than done. Even in the case of double EPW sextics, where the fixed locus has an explicit embedding in $\PP^5$, the properties of $\Fix(\tau)$ are obtained (following Ferretti~\cite{Ferretti:thesis}) by first specializing $X$ to $S^{[2]}$ where $S\subset \PP^3$ is a quartic surface, so that $\Fix(\tau)$ specializes to the surface of bitangent lines to $S$, and then by invoking (non trivial) results of Welters~\cite{Welters:thesis}.  

On the other hand, there are several explicit constructions of HK's $X$ of  $\mathrm{K3}^{[n]}$-type with an antisymplectic involution $\tau$ such that the fixed lattice $H^2(\tau)_{+}$ has rank $2$ (and thus deform in codimension~2 in moduli). One source of examples are moduli spaces of stable sheaves (or more generally, Bridgeland stable objects) on a K3 surface $S$ which is a double cover $S\to\PP^2$: the covering involution of $S$ induces an involution of    the moduli space, provided the Chern character of the sheaves and the stability condition are invariant under the involution of $S$.
Another source of examples are Lagrangian fibrations $X\to\PP^n$ with a fiberwise involution acting as $(-1)$ on $H^1$ of a smooth (abelian) fiber.
In several such examples (covering all $(n,d)$ as above), by making a suitable ``very general'' assumption, we may suppose that $\NS(X)$ has rank $2$, that it equals 
the fixed lattice $H^2(\tau)_{+}$, and that it contains a big class $\lambda$ of square $2$. The class $\lambda$ is not ample. Following~\cite{BM:walls} there is an explicit sequence of flops $X\dra X^{+}$ such that the class $\lambda^{+}\in\NS(X^{+})$ corresponding to $\lambda$ is nef. A divisorial contraction $X^{+}\to X_0$ then produces a singular holomorphic symplectic variety $X_0$ with a Cartier divisor class $\lambda_0$ pulling back to 
$\lambda^{+}$, and an involution $\tau_0$ which corresponds to $\tau$ via the birational map $X\dra X_0$.  We prove that a general smoothing of $(X_0,\lambda_0)$ is a polarized $(X_t,\lambda_t)$ with $q_{X_t}(\lambda_t)=2$ and $\divisore(\lambda_t)=d$, \emph{and} the specialization of the involution $\tau_t$ for $t\to 0$ is equal to $\tau_0$. It follows that the components of $\Fix(\tau_t)$ can be described starting from a description of the components of $\Fix(\tau)$. 

A baby example of the specialization $X_t\to X_0$ is provided by a family of  double covers  $X_t\to \PP^2$  ramified over a sextic curve $\Gamma_t\subset\PP^2$ which is smooth for $t\not=0$ and has a node $p$ for $t=0$ (and no other singularity). The central surface $X_0$ has an ordinary double point. After a base change we  get a smooth filling  
$\wt{\cX}\to T$. The central fiber $\wt{X}_0=X^{+}$ is the minimal desingularization of $ X_0$ and is a double cover $X^{+}\to \Bl_p(\PP^2)\cong\FF_1$. Since we are in dimension $2$ there are no flips, hence $\tau$ is the covering involution of the double cover $X^{+}\to \FF_1$.
Clearly $\Fix(\tau)=\wt{\Gamma}_0$ is the desingularization of the branch curve $\Gamma_0$ of $X_0\to\PP^2$, and $\Fix(\tau_0)$ is  obtained from $\wt{\Gamma}_0$ by joining two  points. This is the simplest example of our specialization in divisibility~1.

The simplest example of our specialization in divisibility~2 is provided by the LLSvS eightfold associated to a cubic fourfold $Y_0\subset\PP^5$ with an ordinary node and not containing a plane.
We sketch the picture and we refer to~\cite{Lehn:singcub} for more details (the results most relevant to us had been developed by Starr in unpublished work).
The construction of the LLSvS variety associated to a smooth cubic not containing a plane works also for $Y_0$ and it produces a singular HK 8-dimensional variety $Z_0$.
The variety parametrizing lines in $Y_0$ containing the node is a (smooth) $K3$ surface $S\subset\PP^4$, i.e., of genus $4$.
As shown in~\cite[Section 6]{Lehn:singcub}, $Z_0$ is birational to the (HK) moduli space $M$ of sheaves $\iota_{C,*}\xi$ where $\iota_C\colon C\hra S$ is the inclusion of a (genus $4$) hyperplane section of $S$ and $\xi$ is a torsion free sheaf on $C$ of degree~$0$ (in~\cite{Lehn:singcub} the degree is $6$, i.e., $\deg \omega_C$; ours is simply a different normalization).
Notice that there is a  Lagrangian fibration $M\to(\PP^4)^{\vee}$ defined by associating to $\iota_{C,*}\xi$ the curve $C$.
The antisymplectic involution on $Z_0$ defines a regular antisymplectic involution of $M$ (a priori it is only birational), namely the fiberwise involution $\tau$ mapping $\iota_{C,*}\xi$ to $\iota_{C,*}\xi^{\vee}$ if $\xi$ is locally free (if $Y_0$ is very general with a node there is a single antisymplectic  birational involution of $M$).
The upshot is that we are reduced to studying the involution $\tau$ on $M$  and the birational map between $M$ and $Z_0$.
We notice that in this case the birational map is the composition of two flops and a divisorial contraction, see Example~\ref{ex:Div2Flopsg48}.
As the dimension increases the birational map becomes more and more complex.

\subsection{Higher divisibility}

There are other examples of antisymplectic involutions deforming in codimension one.
For instance, let $X$ be a projective HK manifold of $\mathrm{K3}^{[n]}$-type and  $\lambda$  an ample class on $X$ such that $q_X(\lambda)=2(n-1)$ and $\div(\lambda)=n-1$.
Then again by \cite[Section 9.1.1]{Markman:Survey} there exists a unique antisymplectic involution $\tau$ of $X$ such that $H^2(\tau)_{+}=\ZZ\lambda$.
However, to study these involutions it seems likely one should choose different degenerations of the involution than those we consider here.

\subsection*{Acknowledgements}
The paper benefited from many useful discussions with the following people which we gratefully acknowledge: Enrico Arbarello, Arend Bayer, Marcello Bernardara, Olivier Debarre, Tommaso de Fernex, Enrico Fatighenti, Alexander Kuznetsov, Giovanni Mongardi, Alexander Perry, Paolo Stellari.
We also thank the referee for useful comments and suggestions.
Parts of the paper were written while the authors were visiting several institutions. In particular, we would like to thank Coll\`ege de France, \'Ecole Normale Sup\'erieure, Institut des Hautes \'Etudes Scientifiques, Universit\'e Paris-Saclay, CIRM in Luminy and MSRI in Berkeley for the excellent working conditions.


\section{Surgery on antisymplectic involutions}\label{sec:surgery}

This is a general section on HK manifolds: we use results by Markman and Namikawa to allow for deformations/smoothings of pairs of HK manifolds (or holomorphic symplectic varieties) together with an involution.

Let $X$ be a compact smooth irreducible hyperk\"ahler manifold (in short, \emph{HK manifold}), meaning that $X$ is a simply-connected compact K\"ahler manifold such that $H^0(X,\Omega^2_X)$ is one-dimensional, spanned by an everywhere non-degenerate holomorphic $2$-form. 
Let $\tau\in\Aut(X)$ be an involution satisfying the following requirements:
\begin{enumerate}[(a)]
\item\label{enum:surgery1} The eigenspace decomposition of the action of $\tau$ on $H^2(X;\QQ)$ is
\begin{equation}\label{autodeco}
H^2(\tau)_{+}=\la \lambda,\delta\ra,\quad H^2(\tau)_{-}=\la \lambda,\delta\ra^{\bot},
\end{equation}
where $\lambda,\delta\in \NS(X)$ are  linearly independent.
\item\label{enum:surgery2} There is a divisorial contraction $\varphi\colon X\to \overline{X}$ of an irreducible divisor $\Delta$ representing the class $\delta$.\footnote{By divisorial contraction we mean that $\overline{X}$ is normal and the relative N\'eron-Severi has rank~1.}
\item\label{enum:surgery3} We have $\lambda=\varphi^{*}(\overline{\lambda})$ for the class $\overline{\lambda}\in H^2(\overline{X})$ of a Cartier divisor on $\overline{X}$.
\item\label{enum:surgery5} The involution $\tau$ descends to an involution $\overline{\tau}\in\Aut(\overline{X})$.
\end{enumerate}

By~\cite[Theorem 2.2]{Namikawa:Def}, it follows that $\Def(\overline{X})$ is smooth and there is a natural branched covering $m\colon \Def(X)\to\Def(\overline{X})$ with branch divisor $B(\overline{X})\subset \Def(\overline{X})$.
By~\cite[Section 3]{markman:modular}, since $\Delta$ is irreducible, the general fiber of $\Delta\to\varphi(\Delta)$ is either a $\PP^1$ or two copies of $\P^1$ intersecting transversally; from now on we assume that we are in the first case.
Hence, by~\cite[Theorem~1.4]{markman:modular} the map $m$ is a double cover.
Moreover 
\begin{equation*}
m^{-1}(B(\overline{X}))=\Def(X,\delta),
\end{equation*}
where the right hand side parametrizes deformations of $X$ which keep $\delta$ a Hodge class.
In fact, by \cite[Lemma 5.1]{Markman:PrimeExceptional} the smooth codimension one subvariety $\Def(X,\delta)$ parametrizes deformations of the pair $(X,\Delta)$.
In particular, $X_u$ and $\overline{X}_{m(u)}$ ($X_u$ and $\overline{X}_{t}$ are the varieties parametrized by $u\in\Def(X)$ and $t\in\Def(\overline{X})$ respectively) are isomorphic for $u$ outside $m^{-1}(B(\overline{X}))$, since $\phi$ is a divisorial contraction.

Note that the involution $\overline{\tau}$ induces an involution on $\Def(\overline{X})$ and on the universal family.  Let  $\Def(\overline{X},\overline{\tau})$ denote the fixed locus of the action on $\Def(\overline{X})$.  There is an induced relative involution on the restriction of the universal family to $\Def(\overline{X},\overline{\tau})$ and so we refer to
$\Def(\overline{X},\overline{\tau})\subset\Def(\overline{X})$ as the locus parametrizing deformations of the pair $(\overline{X},\overline{\tau})$.

\begin{prop}\label{uscire}
Keep notation and hypotheses as above. Then $\Def(\overline{X},\overline{\tau})$ is smooth of codimension $1$ in $\Def(\overline{X})$, and is not contained in  $B(\overline{X})$. 
\end{prop}

\begin{proof} 
Since $\Def(\overline{X})$ is smooth it follows that $\Def(\overline{X},\overline{\tau})$ is smooth as well.

If $\Def(\overline{X},\overline{\tau})=\Def(\overline{X})$, we would have an antisymplectic involution on a general deformation of $X$.
That is impossible, because such an involution would act as multiplication by $(-1)$ on the whole of $H^2(X)$.
Hence $\Def(\overline{X},\overline{\tau})$ has codimension at least $1$ in $\Def(\overline{X})$.

We prove the proposition by comparing $\Def(\overline{X})$ and $\Def(X)$.
The eigenspaces of $H^2(\tau)$ are given by~\eqref{autodeco}.
Since $\tau$ is antisymplectic it follows that the fixed locus of the action of $\tau$ on $H^1(X,\Theta_{X})\cong H^1(X,\Omega_{X})$ is equal to $\la \lambda,\delta\ra^{\bot}$.
Hence $\Def(X,\tau)\subset \Def(X)$ is a (smooth) subspace of codimension $2$.
Since
\begin{equation*}
m(\Def(X,\tau))\subset \Def(\overline{X},\overline{\tau})\cap B(\overline{X}),
\end{equation*}
we get that $\Def(\overline{X},\overline{\tau})\cap B(\overline{X})$ has codimension at most $1$ in $B(\overline{X})$.

Hence in order to finish the proof it suffices to show  that $\Def(\overline{X},\overline{\tau})$ is not contained  in $B(\overline{X})$.
There is a local equation $g$ of $\Def(X,\delta)$ which is a $(-1)$-eigenvalue for the action of $\tau$ on the germ of $\Def(X)$.
Since $m$ is a double cover with ramification divisor $\Def(X,\delta)$, it follows that $g^2$ is the pull-back of an element $h$ of the  germ of $\Def(\overline{X})$ generating the ideal of $B(\overline{X})$.
By compatibility of $\tau$ and $\overline{\tau}$, $h$ is a $(+1)$-eigenvalue for the action of $\overline{\tau}$.
This proves that $\Def(\overline{X},\overline{\tau})$  is not contained in $B(\overline{X})$, and hence it has codimension~1 in $\Def(\overline{X})$.
\end{proof}

For $t\in \Def(\overline{X},\overline{\tau})$ we let $\tau_t\in\Aut(\overline{X}_t)$ be the antisymplectic involution extending $\overline{\tau}$.
If $t \in \Def(\overline{X}) \setminus B(\overline{X})$, then the pair $(\overline{X}_t, \tau_t)$ is a smoothing of $(\overline{X},\overline{\tau})$.

\begin{coro}\label{unaclasse}
There is a line bundle $\cL$ on the family $\cY\to  \Def(\overline{X},\overline{\tau})$ with the following properties:
\begin{itemize}
\item\label{enum:unaclasse1} $\mathrm{c}_1(\cL_0)=\overline{\lambda}$.
\item\label{enum:unaclasse2} If $t\in \Def(\overline{X},\overline{\tau})$, then $\cL_t$ is ample and  $H^2(\tau_t)_{+}=\QQ\cdot  \mathrm{c}_1(\cL_t)$.
\item\label{enum:unaclasse3} For $t\notin B(\overline{X})$, $q_{X_t}(\mathrm{c}_1(\cL_t))=q_X(\lambda)$ and $\divisore(\mathrm{c}_1(\cL_t))=\divisore(\lambda)$.
\end{itemize}
\end{coro}

\begin{proof}
In general, let $X$ be a (smooth) HK manifold and let $\psi\in\Aut(X)$ be an antisymplectic involution.
Then $H^2(\psi)_{+}$ is contained in $H^{1,1}(X)$ and is spanned by its intersection with $H^2(X;\QQ)$.
Since it contains K\"ahler classes (if $\omega$ is a K\"ahler class on $X$ then $\omega+\psi^{*}(\omega)$ is a K\"ahler class in $H^2(\psi)_{+}$), it follows that $H^2(\psi)_{+}$ contains ample classes.
If moreover $\Def(X,\psi)$ has codimension $1$ in $\Def(X)$, then $H^2(\psi)_{+}$ has dimension $1$ and hence it is generated by an ample primitive class.
Applying this to $\tau_t$, for $t\in \Def(\overline{X},\overline{\tau})\setminus B(\overline{X})$, we get the corollary.
\end{proof}

\begin{rema}\label{rmk:FlopInvolution} Notice that the same proof gives similar results also in the case of flopping contractions and compatible involutions.
More precisely, we still start with a HK manifold $X$, but we now assume that $\tau$ is only a birational involution. We further assume that:
\begin{enumerate}[(a)]
\item $H^2(\tau)_+=\la \lambda \ra$, $H^2(\tau)_-=\la \lambda \ra^{\bot}$.
\item There is a birational contraction $\phi\colon X \to \overline{X}$ onto a normal variety and a class $\overline{\lambda}\in H^2(\overline{X})$ of an ample Cartier divisor such that $\lambda=\phi^*(\overline{\lambda})$.
\item The rational involution $\tau$ descends to a regular involution $\overline{\tau}\in \Aut(\overline{X})$.
\end{enumerate}
In this case, again by using~\cite[Theorem~1.4]{markman:modular}, the morphism $m$ gives an isomorphism $m\colon\Def(X)\xrightarrow{\cong}\Def(\overline{X})$.
Then we can deduce directly that $\Def(\overline{X},\overline{\tau})$ is smooth of codimension~1 in $\Def(\overline{X})$, together with the existence of a line bundle $\cL$ on the universal family with the same properties as in Corollary~\ref{unaclasse}.
\end{rema}

\begin{exam}\label{ex:DoubleCoverK3}
Let $\pi\colon S\to\PP^2$ be the double cover ramified over a smooth sextic curve and let $\iota\in\Aut(S)$ be the covering involution of $\pi$.
Then $S$ is a K3 surface and $\iota$ is antisymplectic.
Let $X:=S^{[n]}$ and let $\tau\in\Aut(X)$ be the involution induced by $\iota$, i.e., $\tau([Z]):=[\iota(Z)]$ for $[Z] \in S^{[n]}$.
Then $(X,\tau)$ satisfies the requirements \ref{enum:surgery1}--\ref{enum:surgery5} above.
In fact, let $\overline{X}:=S^{(n)}$ (the $n$-th symmetric power of $S$) and let $\varphi\colon X\to\overline{X}$ be the Hilbert-Chow map.
Then $\varphi$ is the contraction of the divisor $\Delta\subset S^{[n]}$ parametrizing non-reduced subschemes.
Moreover $\overline{\lambda}=h^{(n)}$ where $h:=\pi^{*} \mathrm{c}_1(\cO_{\PP^2}(1))$ and $h^{(n)}\in H^2(S^{(n)};\ZZ)$ is the ``symmetrization'' of $h$.
The involution $\overline{\tau}\in\Aut(\overline{X})$ is defined by $\overline{\tau}(\sum_j m_j x_j)=\sum_j m_j\iota(x_j)$.
By Proposition~\ref{uscire} and Corollary~\ref{unaclasse} the involution $\overline{\tau}$ deforms to involutions $\tau_t\in\Aut(X_t)$ where $X_t$ is smooth and $H^2(\tau_t)_{+}=\QQ\cdot \mathrm{c}_1(\cL_t)$ where $\mathrm{c}_1(\cL_t)$ is a deformation of $h^{(n)}$ (and hence $q_{X_t}(\cL_t)=2$).
Thus $\Fix(\overline{\tau})$ is a specialization of $\Fix(\tau_t)$, and we may hope to determine discrete invariants of $\Fix(\tau_t)$ for general $t$ by examining of $\Fix(\overline{\tau})$.
The latter is a union of irreducible components each isomorphic to the product of two factors, a symmetric power of $\PP^2$ and a symmetric power of the branch curve of $\pi$.
From this description it is not clear that $\Fix(\tau_t)$ is irreducible for general $t$, nor it is easy to determine its birational invariants.
The moral is that one should apply the results in this section to  carefully chosen examples.
\end{exam}


\section{Birational models of Lagrangian fibrations}\label{sec:birational}

In this section we describe the nef and movable cone of certain Lagrangian fibrations with involution in order to study degenerations of involutions induced by ample classes of degree $2$ and to apply the surgery criterion of Section~\ref{sec:surgery}. We introduce the relevant examples in Section~\ref{subsec:examples} and we present a very short review of Bridgeland's theory of stability conditions in Section~\ref{subsec:Bridgeland}.  The two main results are in Section~\ref{subsec:div1}, Lemma~\ref{lem:InvolutionDiv1} (for divisibility~1) and Section~\ref{subsec:div2}, Lemma~\ref{lem:InvolutionDiv2} (for divisibility~2).

Throughout this section $S$ is a smooth projective K3 surface.
We let
\[
\D\colon\Db(S)^{\mathrm{op}}\xlongrightarrow{\cong}\Db(S)\qquad \D(\blank):=\mathbf{R}\cH om_S(\blank,\cO_S)[1]
\]
be the anti-autoequivalence  given by the derived dual composed with the shift functor $[1]$.

\subsection{The examples}\label{subsec:examples}

Let $h$ be the class of an ample divisor $H$ on $S$ and let  $h^2=2g-2$.
We assume that all the curves in the linear system $|\cO_S(H)|$ are integral; for example, this is the case if $\mathrm{NS}(S)=\Z \cdot h$.
We consider the moduli spaces
\[
M_h(0,h,0) \qquad \text{ and }\qquad M_h(0,h,1-g) 
\]
of $h$-semistable pure sheaves on $S$ with Mukai vector $v$ equal to $(0,h,0)$ and $(0,h,1-g)$ respectively. 
An element of $M_h(0,h,0)$ is the isomorphism classe of a sheaf $i_{C,*}(\xi)$, where $i_C\colon C\hra S$ is the inclusion of a curve $C\in |\cO_S(H)|$ and $\xi$ is a torsion free  sheaf of degree $g-1$ on $C$. Similarly, an  element of $M_h(0,h,1-g)$ is the isomorphism classes of a  sheaf $i_{C,*}(\xi)$ where $\xi$ is a torsion free sheaf of degree $0$ on $C$. 
In both cases we have a fibration
\begin{equation}\label{beaufibr}
\begin{matrix}
 M_h(v) & \overset{\pi}{\lra} &  |\cO_S(H)|\cong\PP^g \\
 [F] & \mapsto & \suppdet(F),
\end{matrix}
\end{equation}
where $\suppdet(F)$ is the Fitting support (see~\cite{LePotier:Fitting}).
The fiber of $\pi$ over a smooth curve $C$ is the Jacobian $\Pic^{g-1}(C)$ and $\Pic^{0}(C)$, depending on whether $v=(0,h,0)$ or $v=(0,h,1-g)$. 
We have the involutions induced by the anti-autoequivalences $\D$ and $\cO_S(-H)\otimes\D$ respectively:
\begin{equation}\label{involagr}
\begin{matrix}
 M_h(0,h,0) & \overset{\tau}{\lra} & M_h(0,h,0)  \\
 [F] & \mapsto & [\mathcal{E}xt^1(F,\cO_S)]
\end{matrix}
\qquad\begin{matrix}
 M_h(0,h,1-g) & \overset{\tau}{\lra} & M_h(0,h,1-g)  \\
 [F] & \mapsto & [\mathcal{E}xt^1(F,\cO_S(-H))]
\end{matrix}
\end{equation}
(see~\cite[Lemmas 3.7 \& 3.8]{ASF:Prym}).
Equivalently, when $C$ is a smooth curve, if $[i_{C,*}(\xi)]\in M_h(0,h,0) $ then $\tau([i_{C,*}(\xi)])=i_{C,*}(\omega_C\otimes \xi^{\vee})$, where we set $\xi^{\vee}:=\cH om(\xi,\cO_C)$, while if $[i_{C,*}(\xi)]\in M_h(0,h,1-g)$ then $\tau([i_{C,*}(\xi)])=i_{C,*}(\xi^{\vee})$.

Let $v$ denote one of the two Mukai vectors considered above.
Under our assumptions every sheaf parametrized by $M_h(v)$ is stable, hence $M_h(v)$ is smooth.
In fact $M_h(v)$ is a HK manifold of $\mathrm{K3}^{[g]}$-type  (see~\cite{Muk:Sympl,Huy:bir,Kieran:weight2,Yos:moduli}) and the fibration in~\eqref{beaufibr} is Lagrangian. Let
\[
\vartheta\colon v^{\bot}\xlongrightarrow{\cong} H^2(M_h(v);\Z)
\]
be Mukai's isomorphism, as defined in \cite[(1.6)]{Yos:moduli}.
We recall that $\vartheta$ is an isomorphism of lattices, where $v^{\bot}$ is equipped with the restriction of Mukai's pairing on $H^*(S;\ZZ)$ and of course $H^2(M_h(v);\ZZ)$ is equipped with the Beauville-Bogomolov-Fujiki symmetric bilinear form.
In addition  $\vartheta$ is an isomorphism of Hodge structures ($v^{\bot}$ is a sub-Hodge structure of $H^{*}(S;\Z)$ with Mukai's Hodge structure). Let $H_{\mathrm{alg}}^*(S;\ZZ)$ denote the group of algebraic classes, or equivalently those of type $(1,1)$, in $H^*(S;\ZZ)$. Then the fact that $\vartheta$ is an isomorphism of Hodge structures implies in particular that $\vartheta$ induces
 an isometry between classes in $v^{\bot}_{\mathrm{alg}}:=v^{\bot} \cap H_{\mathrm{alg}}^*(S;\Z)$ and the N\'eron-Severi group $\NS(M_h(v))$.


The result below describes the action of $\tau$ on $H^2(M_h(v);\ZZ)$ in terms of $\vartheta$.

\begin{prop}\label{azione}
Let $v$ be one of the two Mukai vectors considered above and let $\tau\in\Aut(M_h(v))$ be the antisymplectic involution in~\eqref{involagr}.
Suppose that all the curves in the linear system $|\cO_S(H)|$ are integral.
Then  the $H^2(\tau)$-invariant and anti-invariant sublattices in $H^2(M_h(v):\ZZ)$ are given by
\begin{equation}\label{aziopiu}
H^2(\tau)_{+}=
\begin{cases}
\vartheta(\ZZ(1,0,0)+\ZZ(0,0,1)) & \text{if $v=(0,h,0)$,} \\
\vartheta(\ZZ(2,-h,0)+\ZZ(0,0,1)) & \text{if $v=(0,h,1-g)$,}
\end{cases}
\end{equation}
and
\begin{equation}\label{aziomeno}
H^2(\tau)_{-}=\vartheta(H^{2}(S;\ZZ)_{pr}),
\end{equation}
where $H^{2}(S;\ZZ)_{pr}\subset H^2(S;\ZZ)$ is the orthogonal complement of $h$.
\end{prop}

\begin{proof}
The involution is antisymplectic because it preserves the fibers of the Lagrangian fibration $\pi$ and it acts as $(-1)$ on the $H^1$ of a smooth fiber of $\pi$.
Since the decomposition into $H^2(\tau)$-invariant and invariant sublattices is independent of $(S,h)$, we may assume  that $\NS(S)=\Z \cdot h$ and hence 
\begin{equation}\label{nerosev}
\NS(M_h(v))=\vartheta(v^{\bot}\cap (\ZZ(1,0,0)+\ZZ(0,h,0)+\ZZ(0,0,1))).
\end{equation}
Since $\NS(M_h(v))^{\bot}$ is an irreducible Hodge structure containing $H^{2,0}(M_h(v))$, and $\tau$ is antisymplectic, it follows that 
\begin{equation}\label{contmeno}
\vartheta(H^{2}(S;\ZZ)_{pr})\subset H^2(M_h(v);\ZZ)_{-},
\end{equation}
On the other hand $\tau$ sends  $\pi^{*}c_1(\cO_{\PP^g}(1))$ to itself i.e., $\pi^{*}c_1(\cO_{\PP^g}(1))\in H^2(M_h(v);\ZZ)_{+}$. If $\alpha$ is an ample class on $M_h(v)$, then $\alpha+\tau^{*}(\alpha)$ is an invariant ample class on $M_h(v)$. It follows that~\eqref{contmeno} is an equality, i.e.,~\eqref{aziomeno} holds. We also get that the $H^2(\tau)$-invariant sublattice has rank at least $2$, and therefore rank $2$ by~\eqref{contmeno}. Since it is orthogonal to $H^2(\tau)_{-}$ it is equal to the right-hand side of~\eqref{nerosev}. Computing, one gets~\eqref{aziopiu}.
\end{proof}

Our next task is to describe the classes of some key divisors on $M_h(v)$ in terms 
of Mukai's map $\vartheta$. 
First of all, by~\cite[Section 2.3]{LePotier:Fitting}, we have that
\[
f:=\mathrm{c}_1\left(\pi^*\cO_{\PP^g}(1)\right)=\vartheta(0,0,-1)
\]
is the nef divisor associated to the Lagrangian fibration~\eqref{beaufibr}.

The effective divisors appearing in the next two propositions are related to the exceptional loci of the divisorial contractions described in Section~\ref{sec:surgery}.
More precisely, when $v=(0,h,0)$, the divisor equals the contracted divisor, while when $v=(0,h,1-g)$ then the divisor is the proper transform (on a HK manifold birational to $M_h(v)$) of the contracted divisor. 

\begin{prop}\label{prop:Delta}
Let $v=(0,h,0)$ and let $\cD$ be the determinant line bundle on $M_h(v)$ with fiber
\[
\cD_{|[F]}={\bigwedge}^{r} H^0(S,F)^\vee\otimes {\bigwedge}^{r} H^1(S,F)
\]
over $[F]$, where $r:=h^0(S,F)=h^1(S,F)$.
The natural section of $\cD$ is non zero and its divisor $\Delta$ has the following properties:
\begin{eqnarray}
\supp\Delta & = & \{[F]\in M_h(v) \mid h^0(S,F)>0\}, \label{supposta} \\
\cl(\Delta) & = & \vartheta(1,0,1). \label{classedel}
\end{eqnarray}
\end{prop}

\begin{proof}
Equation~\eqref{supposta} is the zero locus of the natural section of $\cD$ and it is not equal to the whole of $M_h(v)$ because if $C\in|\cO_S(H)|$ is smooth its restriction to $\pi^{-1}(C)=\Pic^{g-1}(C)$ is the natural $\Theta$ divisor. Equation~\eqref{classedel} follows from the Grothendieck-Riemann-Roch Theorem.
\end{proof}

\begin{rema}\label{delta=theta}
The divisor $\Delta$ defined above is birational to the relative $(g-1)$-fold symmetric product of the universal family of curves over $|\cO_S(H)|$.
Since $\dim |\cO_S(H)|=g$, it is birational to a $\P^1$-bundle over $S^{[g-1]}$ with generic fiber ruling the pencil of curves through $g-1$ general points. 
\end{rema}

Assume that $g$ is even. The analogue of $\Delta$ on $M_h(0,h,1-g)$ is defined starting from the unique stable vector bundle $A$ on $S$ with Mukai vector
\begin{equation*}
v(A):=\ch(A).\sqrt{\td_S}=\left(2,-h,\frac g2\right).
\end{equation*}
The existence and uniqueness of $A$ are proved in~\cite{Muk:K3}.
Explicitly, $A$ is a Lazarsfeld--Mukai bundle obtained as an elementary modification
\begin{equation}\label{eq:eccoA}
0 \to A \to \cO_S^{\oplus 2} \to i_{C,*}(\xi) \to 0,
\end{equation}
where $C\in|\cO_S(H)|$ is a smooth curve and $\xi$ is a line bundle of degree $g/2+1$ with $h^0(C,\xi)=2$.

\begin{prop}\label{prop:Deltadiv2}
Suppose that $\NS(S)=\ZZ h$, $g$ is even, and $g\ge 4$.
Let $v=(0,h,1-g)$ and let $\cD$ be the determinant line bundle on $M_h(v)$ with fiber
\[
\cD_{|[F]}={\bigwedge}^{r} H^0(S,A^{\vee}\otimes F)^\vee\otimes {\bigwedge}^{r} H^1(S,A^{\vee}\otimes F)
\]
over $[F]$, where $r:=h^0(S,A^{\vee}\otimes F)=h^1(S,A^{\vee}\otimes F)$.
The natural  section of $\cD$ is non zero and its divisor $\Delta$ has the following properties:
\begin{eqnarray}
\supp\Delta & = & \{[ F]\in M_h(v) \mid h^0(S,A^{\vee}\otimes F)>0\}, \label{suppostina} \\
\cl(\Delta) & = & \vartheta\left(2,-h,\frac g2\right). \label{classedelina}
\end{eqnarray}
Moreover if $C\in|\cO_S(H)|$ is smooth the restriction of $\Delta$ to $\pi^{-1}(C)=\Pic^{0}(C)$ is  twice the principal polarization.
\end{prop}

\begin{proof}
Equation~\eqref{suppostina} is the zero locus of the natural section of $\cD$. The fact that it is not equal to the whole of $M_h(v)$ follows from the more general results we will prove in Section~\ref{subsec:div2} (and later on in Section~\ref{sec:ExplicitMMP}, we will also need a more precise vanishing result).
We give here in any case a sketch of the classical argument, for completeness.
Let $C\in|\cO_S(H)|$ be smooth. 
The zero locus of the restriction of the section to $\pi^{-1}(C)=\Pic^{0}(C)$ is equal to
\begin{equation}\label{doublepol}
\{[\xi]\in \Pic^{0}(C) \mid h^0(C,A^{\vee}_{|C}\otimes\xi)>0\}.
\end{equation}
It suffices to prove that the above set is not the whole of $\Pic^{0}(C)$.
By Proposition 1.6.2. in~\cite{Raynaud} this holds if (and only if) the restriction of $A^{\vee}$ to $C$ is semistable.
This can be directly proved as follows.

Suppose that $A^{\vee}_{|C}$ is not semistable and let $A^{\vee}_{|C}\to Q$ be a destabilizing quotient.
Let $E$ be the locally free sheaf on $S$ fitting into the exact sequence
\begin{equation*}
0\lra E\lra A^{\vee}\lra i_{C,*}(Q)\lra 0.
\end{equation*}
We can compute easily that $4c_2(E)-c_1(E)^2\le -2g+4<0$, and so by the Bogomolov inequality it follows that $E$ is slope unstable.
Since $c_1(E)=0$ and $\NS(S)=\ZZ h$ there exists 
a non zero map $\cO_S(mH)\to E$ for some $m>0$.
Composing with the inclusion $E\subset A^{\vee}$ we get a non zero map $\cO_S(mH)\to A^{\vee}$ contradicting the stability of $A^{\vee}$.

This proves that the set in~\eqref{doublepol} is not the whole of $\Pic^{0}(C)$. Actually the set in~\eqref{doublepol} is the support of a cycle which is  twice the principal polarization of $\Pic^{0}(C)$ (see Proposition 1.8.1 in~\cite{Raynaud}), and which is the restriction of $\Delta$ to  $\Pic^{0}(C)$.
Lastly Equation~\eqref{classedelina} follows again from the Grothendieck-Riemann-Roch Theorem.
\end{proof}

\begin{exam}
We recall that the LLSvS variety of a cubic fourfold with a node is birational to $M_h(0,h,-3)$ (see the discussion at the end of Section~\ref{subsec:pazzaidea}).
An explicit description of an open non empty subset of $\Delta\subset M_h(0,h,-3)$ can be given as follows.
Let $C\in|\cO_S(H)|$ be a general smooth curve.
There are two distinct $g^1_3$'s on $C$, say $D_1,D_2$, and one has the equality of divisors on $C$:
\begin{equation*}
\Delta_{|\Pic^{0}(C)}  =  \{[\xi]\in \Pic^{0}(C) \mid h^0(C,D_1\otimes\xi)>0\}+ \{[\xi]\in \Pic^{0}(C) \mid h^0(C,D_2\otimes\xi)>0\}.
\end{equation*}
For general $g$ we will not have a similar explicit description in terms of linear series, which is the reason we need to look at higher rank vector bundles and thus use Bridgeland stability to deal with them on $S$.
\end{exam}

\subsection{Review of Bridgeland stability conditions}\label{subsec:Bridgeland}

Stability conditions on derived categories were defined by Bridgeland in \cite{Bridgeland:Stab}; the definition was extended to include existence of moduli spaces and to also work  in families in \cite{KS:stability} and \cite{BLMNPS:family}. We briefly recall here the definition given in \cite[Definition 2.1 \& Remark 2.2]{MS:Mumford} in the case of K3 surfaces.

\begin{defi}\label{def:Bridgeland}
A \emph{Bridgeland stability condition} on $\Db(S)$ is a pair $\sigma=(Z,\mathcal{P})$ where
\begin{itemize}
    \item $Z\colon H^*_\mathrm{alg}(S,\ZZ) \to \C$ is a group homomorphism, called \emph{central charge}, and
    \item $\mathcal{P}=\cup_{\phi\in\R}\mathcal{P}(\phi)$ is a collection of full additive subcategories $\mathcal{P}(\phi)\subset \Db(S)$
\end{itemize}
satisfying the following conditions:
\begin{enumerate}[(a)]
    \item for all nonzero $E\in\mathcal{P}(\phi)$ we have $Z(v(E))\in \R_{>0}\cdot e^{i\pi\phi}$;
    \item\label{item:shift} for all $\phi\in\R$ we have $\mathcal{P}(\phi+1)=\mathcal{P}(\phi)[1]$;
    \item if $\phi_1>\phi_2$ and $E_j\in\mathcal{P}(\phi_j)$, then $\Hom(E_1,E_2)=0$;
    \item (Harder--Narasimhan filtrations) for all nonzero $E\in\Db(S)$ there exists a finite sequence of morphisms
    \[
    0=E_0 \xrightarrow{s_1} E_1 \xrightarrow{s_2}\dots \xrightarrow{s_m} E_m=E
    \]
    such that the cone of $s_j$ is in $\mathcal{P}(\phi_j)$ for some sequence $\phi_1>\phi_2>\dots>\phi_m$ of real numbers;
    \item\label{eq:supportproperty} (support property) there exists a quadratic form $Q$ on the vector space $H^*_\mathrm{alg}(S;\mathbb{R})$ such that
    \begin{itemize}
        \item the kernel of $Z$ is negative definite with respect to $Q$, and
        \item for all $E\in\mathcal{P}(\phi)$ for any $\phi$ we have $Q(v(E))\geq0$;
    \end{itemize}
    \item\label{eq:openness} (openness of stability) for every scheme $T$ and every $T$-perfect complex $E\in\mathrm{D}_{T\text{-perf}}(S\times T)$ the set
    \[
    \left\{t\in T \, :\, E_t\in \mathcal{P}(\phi) \right\}
    \]
    is open;
    \item\label{eq:boundedness} (boundedness) for any $v\in\Lambda$ and $\phi\in\R$ such that $Z(v)\in \R_{>0}\cdot e^{i\pi\phi}$ the functor
    \[
    T \mapsto \mathfrak{M}_{\sigma}(v,\phi)(T):=\left\{E\in\mathrm{D}_{T\text{-perf}}(S\times T)\,:\, E_t\in\mathcal{P}(\phi)\text{ and }v(E_t)=v, \text{ for all } t\in T \right\}
    \]
    is bounded.
\end{enumerate}
\end{defi}

The objects in $\mathcal{P}(\phi)$ are called \emph{$\sigma$-semistable} of phase $\phi$. The simple objects in the abelian category $\mathcal{P}(\phi)$ are called \emph{$\sigma$-stable}; Jordan-H\"older filtrations exist and $S$-equivalence classes of $\sigma$-semistable objects are then defined accordingly. The phases of the first and last factor in the Harder--Narasimhan filtration of an object $E$ are denoted $\phi^+(E)$ and $\phi^-(E)$.

Definition~\ref{def:Bridgeland} can be reformulated also in terms of slope.
More precisely, the extension-closed category $\mathcal{A}:=\mathcal{P}((0,1])$ generated by all semistable objects with phases in the interval $(0,1]$ is an abelian category, which is the heart of a bounded t-structure on $\Db(S)$.
The real and imaginary parts of the central charge $Z$ behave like a degree and rank function on $\mathcal{A}$: for a nonzero object $E\in\mathcal{A}$, $\Im Z(E)\geq0$ and if $\Im Z(E)=0$, then $\mathrm{Re}\, Z(E) <0$.
An object $E\in\mathcal{A}$ is $\sigma$-semistable if and only if it is slope-semistable with respect to the slope $\mu_\sigma:=-\frac{\mathrm{Re}\, Z}{\Im Z}$.
Then by property~\ref{item:shift} all objects in $\mathcal{P}$ are shifts of semistable objects in $\mathcal{A}$.
The converse is also true.
Let $Z$ be a central charge on the heart of a bounded t-structure $\mathcal{A}$ satisfying the above numerical properties.
We define (semi)stable objects to be the slope-(semi)stable objects in $\mathcal{A}$, together with their shifts in $\Db(S)$.
We obtain a stability condition in $\Db(S)$ once Harder--Narasimhan filtrations exist in $\mathcal{A}$ and the remaining properties \ref{eq:supportproperty}, \ref{eq:openness}, \ref{eq:boundedness} are satisfied. When we want to stress the category $\mathcal{A}$ in the definition of stability we use the notation $\sigma=(Z,\mathcal{A})$.

By Bridgeland's Deformation Theorem \cite[Theorem 1.2]{Bridgeland:Stab}, the set $\Stab(\Db(S))$ of stability conditions is a complex manifold, when endowed with the coarsest topology such that the functions $E \mapsto \phi^+(E), \phi^-(E)$, for $E \in \Db(S)$, and $\mathcal{Z}\colon\Stab(\Db(S))\to \Hom(H^*_\mathrm{alg}(S;\C),\C)$ defined by $(Z,\mathcal{P})\mapsto Z$ are continuous.
We denote by $\Stab^\dagger(\Db(S))$ the connected component of the space of stability conditions on  $\Db(S)$ described in \cite{Bridgeland:StabK3}. The map $\mathcal{Z}$ on  $\Stab^\dagger(\Db(S))$ is a topological cover of its image, and the latter is explicitly described in \cite[Theorem~1.1]{Bridgeland:StabK3}.

\begin{exam}\label{ex:PicardRank1Bridgeland}
Let us assume that $(S,h)$ is a polarized K3 surface with $\NS(S)=\Z \cdot h$.
In this case, by \cite[Theorem~1.3]{BayerBridgeland:StabK3}, $\Stab^\dagger(\Db(S))$ is simply connected and the map $\mathcal{Z}$ is the universal cover of its image.
An open subset
\[
\Stab^{\mathrm{geom}}(\Db(S))\subset \Stab^\dagger(\Db(S))
\]
can be explicitly defined as follows.
Let $\alpha,\beta\in\R$, $\alpha>0$.
We let $\sigma_{\alpha,\beta}:=(Z_{\alpha,\beta},\coh^\beta(S))$, where
\[
Z_{\alpha,\beta}\colon H^*_\mathrm{alg}(S,\Z) \to \C, \qquad v \mapsto (e^{(\beta+i\alpha)h},v) 
\]
and
\[
\coh^\beta(S) := \langle \mathcal{T}^\beta, \mathcal{F}^\beta[1] \rangle  = \left\{ E\in\Db(S)\,:\,\begin{array}{l} H^i(E)=0, \text{ for all }i\neq 0,-1\\ H^{-1}(E)\in \mathcal{F}^\beta \\ H^0(E)\in \mathcal{T}^\beta \end{array}\right\},
\]
where, for a complex $E\in\Db(S)$, $H^i(E)$ denotes its $i$-th cohomology sheaf and
\begin{align*}
    \mathcal{T}^{\beta} &:= \{E \in \coh(S) : \ \text{all quotients $E \thra Q$ satisfy $\mu(Q) > \beta$}\}, \\
    \mathcal{F}^{\beta} &:= \{E \in \coh(S) : \ \text{all non-trivial subobjects $K \hra E$ satisfy $\mu(K) \leq \beta$}\}.
\end{align*}

By \cite[Lemma 6.2]{Bridgeland:StabK3} and \cite[Theorem 1.4]{Toda:ModuliK3}, the pair $\sigma_{\alpha,\beta}$ gives a stability condition on $\Db(S)$ if and only if $Z_{\alpha,\beta}(\delta)\notin\R_{\leq0}$, for all $\delta\in H^*_\mathrm{alg}(S,\Z)$ with $\delta^2=-2$.
The open subset $\Stab^{\mathrm{geom}}(\Db(S))$ consists of all the orbits of the stability conditions $\sigma_{\alpha,\beta}$ by the action of the group $\widetilde{\mathrm{GL}}_2^+(\R)$, the universal cover of the group of $2\times 2$ real matrices with positive determinant. It can also be characterized as the subset of $\Stab(\Db(S))$ consisting of those stability conditions where the skyscraper sheaves of length one are all stable of the same phase.
The whole component $\Stab^\dagger(\Db(S))$ is the union of all the orbits of the closure $\overline{\Stab^{\mathrm{geom}}(\Db(S))}$ with respect to the action of the group of autoequivalences of $\Db(S)$.
\end{exam}

Let $v\in H^*_\mathrm{alg}(S,\Z)$ be a primitive Mukai vector.
For a stability condition $\sigma=(Z,\mathcal{P})\in\Stab^\dagger(\Db(S))$, we denote by $M_\sigma(v,\phi)$ the moduli space parametrizing $S$-equivalence classes of $\sigma$-semistable objects in $\mathcal{P}(\phi)$ with Mukai vector $v$.
As mentioned above, if the phase is fixed, by shifting we can always assume that all $\sigma$-semistable objects lie in $\mathcal{A}$ and have Mukai vector $\pm v$.
Hence, by abuse of notation, we will write $M_\sigma(v)$ and forget the phase.\footnote{Equivalently, we could have defined $M_\sigma(v)$ to denote the moduli space parametrizing $S$-equivalence classes of $\sigma$-semistable objects in $\mathcal{A}$ with Mukai vector $\pm v$.}
We denote by $M_\sigma^\mathrm{st}(v)$ the open subspace parametrizing $\sigma$-stable objects.

The main result of \cite{BM:projectivity} establishes that for any such choice of stability condition $\sigma\in \Stab^\dagger(\Db(S)) $, there is a piece-wise analytic continuous map
\begin{equation}\label{eq:ell defn}
\ell\colon \Stab^\dagger(\Db(S))\rightarrow NS(M_\sigma(v))
\end{equation}
whose image is the intersection of the movable cone with the big cone of $M_\sigma(v)$. In particular, there is a naturally defined divisor class $\ell_\sigma$ on $M_\sigma(v)$; for the stability conditions $\sigma_{\alpha,\beta}$ in Example~\ref{ex:PicardRank1Bridgeland}, the class of this divisor has been computed in \cite[Lemma 9.2]{BM:projectivity}.

We have the following result, summarized in \cite[Theorem 21.24 \& Theorem 21.25]{BLMNPS:family}, and based on previous work in \cite{Lieblich:moduli,AP:tstructures,Toda:ModuliK3,BM:projectivity,AHLH:Moduli}:

\begin{theo}\label{thm:Moduli}
Let $v\in H^*_\mathrm{alg}(S,\Z)$ be a Mukai vector with $v^2\geq-2$.
Then for any $\sigma\in\Stab^\dagger(\Db(S))$ the moduli space $M_\sigma(v)$ is a non-empty proper algebraic space and the divisor class $\ell_\sigma$ is strictly nef.
If $v$ is primitive and $\sigma$ is generic with respect to $v$, then $M_\sigma(v)=M_\sigma^\mathrm{st}(v)$ is smooth projective integral hyperk\"ahler manifold of dimension $v^2+2$ deformation equivalent to the Hilbert scheme of points on a K3 surface and the class $\ell_\sigma$ is ample.
\end{theo}

Mukai's isomorphism $\vartheta\colon v^\perp\xrightarrow{\cong}H^2(M_\sigma(v);\ZZ)$ holds more generally for moduli spaces $M_\sigma(v)$, when $v$ is primitive and $\sigma$ generic with respect to $v$. By abuse of notation, we will sometimes directly identify $v^\perp$ with $H^2(M_\sigma(v);\ZZ)$, forgetting the map $\vartheta$.

We will also need the main result of \cite{BM:walls} in order to study nef and movable cones of moduli spaces (the case of movable cones, Part~\ref{enum:walls2} below, was proved earlier and in greater generality in \cite[Lemma 6.22]{Markman:Survey}). Recall that the the positive cone $\mathrm{Pos}(M)$ of a HK manifold $M$ is the connected component containing an ample divisor of the set $\{ D\in\NS(M)_\R\,:\, q_M(D)>0\}$.

\begin{theo}\label{thm:walls}
Let $v\in H^*_\mathrm{alg}(S,\Z)$ be a primitive Mukai vector with $v^2\geq2$ and let $\sigma_0\in\Stab^\dagger(\Db(S))$ be a generic stability condition with respect to $v$.
Let $M:=M_{\sigma_0}(v)$.
\begin{enumerate}[(i)]
\item\label{enum:walls1} All smooth projective HK manifolds which are birational to $M$ are isomorphic to $M_{\sigma}(v)$, for some $\sigma\in\Stab^\dagger(\Db(S))$ generic with respect to $v$, and conversely each $M_\sigma(v)$ is birational to $M$.
\item\label{enum:walls2} The interior of the movable cone of $M$ is the connected component of
\[
\mathrm{Pos}(M) \setminus \bigcup_{\begin{subarray}{c}{a \in H^*_\mathrm{alg}(S;\Z)\ \text{\rm  s.t.}}\\ {\text{\rm either } a^2=-2\ \text{\rm and } (a,v)=0,} \\ {\text{\rm or } a^2=0\ \text{\rm and } (a,v)=1,2}\end{subarray}} \vartheta(a^\bot),
\]
containing an ample divisor. 
\item\label{enum:walls3} The ample cones of each birational model $M_{\sigma}(v)$ can be identified with the connected component of
\[
\mathrm{Pos}(M) \setminus \bigcup_{\begin{subarray}{c}{a \in H^*_\mathrm{alg}(S;\Z)\ \text{\rm  s.t.}}\\ {a^2\geq -2\ \text{\rm and }} \\ {0\leq (a,v) \leq v^2/2}\end{subarray}} \vartheta(a^\bot)
\]
containing an ample divisor on $M_\sigma(v)$.
\end{enumerate}
\end{theo}

The connection between Part~\ref{enum:walls1} and Parts~\ref{enum:walls2} and \ref{enum:walls3} in Theorem~\ref{thm:walls} is via the map $\ell$ in \eqref{eq:ell defn}. More precisely, this map glues to a piece-wise analytic continuous map $\ell\colon\Stab^\dagger(\Db(S))\to \NS(M)_\R$. Then any \emph{wall} in the positive cone of M, namely the codimension~1 boundary components of the above chamber decomposition, is given by a \emph{wall} in $\Stab^\dagger(\Db(S))$, namely the codimension~1 components of the locus where semistable objects of Mukai vector $v$ change. Conversely, a wall $\mathcal{W}$ in $\Stab^\dagger(\Db(S))$ gives a wall in the positive cone if and only if, for $\overline{\sigma}\in\mathcal{W}$, if we write $\ell_{\overline{\sigma}}=\vartheta(w_{\overline{\sigma}})$, then $w_{\overline{\sigma}}$ is orthogonal to $a\in H^*_\mathrm{alg}(S;\Z)$ with $a^2\geq -2$ and $0\leq (a,v) \leq v^2/2$.

The following two examples, which are applications of Theorem~\ref{thm:walls}, will be used to study explicitly the exceptional loci of the wall--crossing flops.

\begin{exam}[{\cite[Example 9.7]{BM:projectivity}}]\label{ex:rk1}
Let $(S,h)$ be a polarized K3 surface with $\NS(S)=\Z \cdot h$, $h^2=2g-2$.
Let $v\in H^*_\mathrm{alg}(S,\Z)$ be a primitive Mukai vector, $v=(r_0,c_0\cdot h,s_0)$, with $r_0,c_0,s_0\in\Z$, $r_0\geq0$.
We assume there exist $x,y\in\Z$, $x>0$, such that $xc_0-y r_0=1$.
We define $\alpha_0>0$ as
\[
\alpha_0 := \begin{cases}\frac{1}{x\sqrt{g-1}}, \text{ if } \frac{(g-1) y^2 + 1}{x}\in\Z \\ 0, \text{ otherwise.} \end{cases}
\]
and we consider the stability conditions $\sigma_{\alpha,y/x}$ in $\Stab^{\mathrm{geom}}(\Db(S))$, for $\alpha>\alpha_0$.

Then there is no wall for $M_{\sigma_{\alpha,y/x}}(v)$ 
for all $\alpha>\alpha_0$.
Indeed, the imaginary part of $Z_{\alpha,y/x}$ has the property that
\[
\frac{1}{\alpha}\cdot \Im Z_{\alpha,y/x}(r,c\cdot h, s) = (2g-2) \frac{xc-yr}{x} \in (2g-2)\frac{1}{x}\cdot \Z
\]
and $(1/\alpha)\Im Z_{\alpha,y/x}(v_0)= (2g-2)/x$ is the smallest possible positive value.

As a consequence, $\ell_{\sigma_{\alpha,y/x}}$ is ample on $M_{\sigma_{\alpha,y/x}}(v)$ for all $\alpha> \alpha_0$.
In particular, if we denote the constant moduli space $M:=M_{\sigma_{\alpha,y/x}}(v)$, we have that
\[
\R_{\geq0} \left\{\ell_{\sigma_{\alpha,y/x}}\, : \, \alpha>\alpha_0\right\} \subset \mathrm{Amp}(M).
\]
\end{exam}

\begin{exam}\label{ex:rk2}
In the notation of Example~\ref{ex:rk1}, we assume now there exist $x,y\in\Z$, $x>0$, such that $xc_0-y r_0=2$.
In this case, there may be walls, but since $(1/\alpha)\Im Z_{\alpha,y/x}(v_0)= 2(2g-2)/x$, if we denote by $v_1$ the Mukai vector of a destabilizing subobject or quotient, we can only have $(1/\alpha)\Im Z_{\alpha,y/x}(v_1)=(2g-2)/x$. Hence, if there exists $\alpha_1>\alpha_0$ such that $E\in M_{\sigma_{\alpha_1,y/x}}(v_0)$ is not stable, we must have an exact sequence
\[
0 \to F \to E \to Q \to 0
\]
with $(1/\alpha)\Im Z_{\alpha,y/x}(F)=(2g-2)/x$ and $(1/\alpha)\Im Z_{\alpha,y/x}(Q)=(2g-2)/x$. Thus, $F$ and $Q$ are stable for all $\alpha>\alpha_0$, by Example~\ref{ex:rk1}.
In particular, the birational transformation induced by such a wall at $\alpha_1$ is a Mukai flop whose exceptional locus is the projective bundle (in the analytic topology) $\PP_{Z}(\mathcal{V})$, where
\[
Z = M_{\sigma_{\alpha_1,y/x}}(v(F)) \times M_{\sigma_{\alpha_1,y/x}}(v(Q))
\]
and $\mathcal{V}$ is the twisted vector bundle of rank $(v(F),v(Q))$ induced by the relative Ext-sheaf.\footnote{$\mathcal{V}$ is an actual vector bundle when universal families exist on $M_{\sigma_{\alpha_1,y/x}}(v(F))$ and $M_{\sigma_{\alpha_1,y/x}}(v(Q))$.}
\end{exam}

Finally, we recall that, under the assumptions of Example~\ref{ex:rk1} and Example~\ref{ex:rk2}, the moduli space $M_{\sigma_{\alpha,y/x}}(v)$ is isomorphic to a moduli space of $h$-semistable sheaves, for $\alpha\gg\alpha_0$ by \cite[Proposition 14.2]{Bridgeland:StabK3} and \cite[Section 6.2]{Toda:ModuliK3}.

\subsection{Divisibility 1}\label{subsec:div1}
In this section we follow mostly~\cite{Bayer:BN}, but see also~\cite{Mar:BN,Yos:Mukai,Yos:BN,ArcaraBertram:BridgelandKtrivial}.
Let $(S,h)$ be a polarized K3 surface of genus $g$ with $\NS(S)=\Z \cdot h$.

We let $v = (0,h,0)$ and $M:=M_h(v)$.
Let
\begin{equation*}
f = \vartheta(0,0,-1),\quad \delta= \vartheta(1,0,1)=\cl(\Delta),
\quad \lambda = \vartheta(1,0,-1)=2f+\delta,
\end{equation*}
where $\Delta$ is the divisor defined in Proposition~\ref{prop:Delta}.
Notice that $\lambda$ has square $2$ and divisibility $1$.
Recall from Section~\ref{subsec:examples} that the class $f$ induces the Lagrangian fibration
\[
\pi\colon M \to |\cO_S(H)|\cong \P^g,
\]
and is thus one of the two rays of the nef cone.
We can determine the other ray of the nef cone as well.

\begin{lemm}\label{lem:NefDiv1}
The nef and movable cone coincide for $M$:
\[
\Nef(M) = \Mov(M) = \R_{\geq0} f+ \R_{\geq0}\lambda.
\]
\end{lemm}

\begin{proof}
We want to apply Theorem~\ref{thm:walls}.
This is an explicit computation with Pell's equations, but it can also be easily done by using Example~\ref{ex:rk1}: indeed $v=(0,h,0)$ satisfies the assumptions, with $x=1$ and $y=0$.
By \cite[Lemma 9.2]{BM:projectivity}, the associated divisor class $\ell_{\sigma_{\alpha,0}}$ on $M$ satisfies 
\begin{align*}
&\ell_{\sigma_{\alpha,0}} \to f,\qquad \text{for } \alpha \to \infty\\
&\ell_{\sigma_{\alpha,0}} \to \lambda,\qquad \text{for } \alpha \to \frac{1}{\sqrt{g-1}},
\end{align*}
up to positive multiplicative constants.
Hence
\[
\Nef(M) \supset \R_{\geq0} f+ \R_{\geq0}\lambda.
\]
But $\delta^2=-2$ and $(\delta,\lambda)=0$. Hence, by Theorem~\ref{thm:walls}, $\lambda$ is not in the interior of the movable cone of $M$ and
\[
\Mov(M) \subset \R_{\geq0} f+ \R_{\geq0}\lambda,
\]
finishing the proof.
\end{proof}

Lemma~\ref{lem:NefDiv1} gives that the effective cone of  $M$ is $\mathbb{R}_{\ge0} f + \mathbb{R}_{\ge0}\delta$.
Since $\delta$ is indivisible and we are assuming that $\NS(S)=\Z \cdot h$, we notice the following consequence about the divisor $\Delta$.

\begin{coro}\label{coro:NefDiv1}
The divisor $\Delta$ on $M$ is prime i.e., reduced and irreducible.
\end{coro}

Notice that the proof of Lemma~\ref{lem:NefDiv1} gives a result slightly stronger than the thesis, namely that all objects in $M$ are stable for all stability conditions $\sigma_{\alpha,0}$, $\alpha>\frac{1}{\sqrt{g-1}}$.
Next, we describe the divisorial contraction induced by $\lambda$ (see~\cite[Corollary 6.7]{Bayer:BN}).

\begin{lemm}\label{lem:DivContractionDiv1}
The class $\lambda$ induces a divisorial contraction
\[
\phi\colon M \to \overline{M}
\]
with exceptional divisor $\Delta$.
\end{lemm}

\begin{proof}
We use the notation from \cite[Section 5]{BM:walls}.
By~\cite[Proposition 5.1]{BM:walls} to each wall in the space of stability conditions there is an associated rank~2 primitive hyperbolic lattice $\mathcal{H}\subset H^*_\mathrm{alg}(S,\Z)$.
In our case, for the wall associated to the divisorial contraction $\varphi$, and hence to the class $\lambda$, this lattice is given by
\[
\mathcal{H}=\lambda^\perp = \Z\delta +\Z v.
\]
The behavior of $\varphi$ is determined by the existence of certain classes in $\mathcal{H}$ ~\cite[Theorem~5.7]{BM:walls}). More precisely, we are looking for classes $a\in\mathcal{H}$ such that $a^2\geq-2$ and $0\leq (a,v)\leq g-1$.
It is immediate to see that the only possibility is when $a=\pm \delta$.
Hence, the exceptional locus of the divisorial contraction $\phi$  is 
$\Delta$.
\end{proof}

Next we examine a stratifcation of $\Delta$.
Given $k\geq 1$, set
\begin{equation*}
\Delta(k) := \left\{ F \in M\, :\, h^0(S,F)=k \right\}.
\end{equation*}
Thus $\Delta$ is the disjoint union of the $\Delta(k)$'s.
A more precise result is the following.

\begin{lemm}\label{lemm:Delstratdiv1}
We have
\begin{equation}\label{stratdel}
\Delta=\bigsqcup\limits_{1\le k\le \lfloor \sqrt{g} \rfloor}\Delta(k),\qquad \ov{\Delta(i)}=\bigsqcup\limits_{i\le k\le \lfloor \sqrt{g}\rfloor}\Delta(k).
\end{equation}
Moreover let $\overline{\sigma}\in \Stab^{\mathrm{geom}}(\Db(S)) \subset\Stab^\dagger(\Db(S))$ be a stability condition such that $\ell_{\overline{\sigma}}=\lambda$. Then each $\Delta(k)$ is isomorphic to a locally trivial (in the analytic topology) Grassmannian bundle $\Gr(k,\cU_{2k})$ over the $(2g-2k^2)$--dimensional moduli space $M^\mathrm{st}_{\overline{\sigma}}(b_k)$ of \emph{stable} objects where $b_k = - (k, -h, k)$ and where $\cU_{2k}$ is a twisted vector bundle of rank $2k$ induced by the relative Ext-sheaf.
\end{lemm}

\begin{proof}
We describe each $\Delta(k)$ as a Grassmannian bundle over $M^\mathrm{st}_{\overline{\sigma}}(b_k)$. 
Explicitly, if we choose $\overline{\sigma}=\sigma_{\alpha,\beta}$ with $(\alpha,\beta)$ sufficiently close to $(1/\sqrt{g-1},0)$ and $\beta<0$, then $F\in \Delta(k)$ if and only if its Jordan--H\"older filtration with respect to $\overline{\sigma}$ is
\begin{equation}\label{eq:DivContractionDiv1}
\cO_S\otimes W^\vee \to F \to F_k,
\end{equation}
where $F_k\in M^\mathrm{st}_{\overline{\sigma}}(b_k)$, $\dim W=k$, $U_{2k}:=\Hom(F_k,\cO_S[1])$, and where $W$ is viewed as a subspace of $U_{2k}$ by applying $\Hom(\blank, \cO_S)$ to the triangle above and using the fact that $\Hom(F,\cO_S)=0$. The twisted vector bundle $\cU_{2k}$ has fiber equal to $U_{2k}$ over the point $F_k\in M^\mathrm{st}_{\overline{\sigma}}(b_k)$. The upper bound $k\le \lfloor \sqrt{g}  \rfloor$ follows because the moduli space has dimension $(2g-2k^2)$, and the second equation in~\eqref{stratdel} holds because each stratum has the expected codimension, as determinantal subvariety.
\end{proof}

Finally, we look at the involution $\tau$ on $M$; recall from Section \ref{subsec:examples} that it is induced by the functor $\D$ and acts fiberwise with respect to the Lagrangian fibration $\pi$. 
Let $\mathrm{ST}_{\cO_S}\colon\Db(S)\xrightarrow{\cong}\Db(S)$ be the spherical twist at $\cO_S$.

\begin{lemm}\label{lem:InvolutionDiv1}
We keep the notation of Lemma~\ref{lem:DivContractionDiv1}.
The involution $\tau$ on $M$ maps each  $\Delta(k)$ to itself and is compatible with the Grassmannian bundle $\Delta(k) \to M^\mathrm{st}_{\overline{\sigma}}(b_k)$, with the involution on $M^\mathrm{st}_{\overline{\sigma}}(b_k)$ induced by the anti-autoequivalence $\Phi\colon\Db(S)^{\mathrm{op}}\xrightarrow{\cong}\Db(S)$ given by $\Phi:=\mathrm{ST}_{\cO_S}\circ\D$.
In particular $\tau$  induces a regular involution $\ov{\tau}\colon\ov{M}\xrightarrow{\cong}\ov{M}$.
\end{lemm}

\begin{proof}
Let $F\in\Delta(k)$.
Since $\chi(\cO_S,F)=0$, the dimensions of $\Hom(\cO_S,F)$ and $\Hom(\cO_S,\D(F))$ are the same.
Hence, by definition, $\tau$ maps each $\Delta(k)$ to itself.

To describe explicitly the involution induced on $M^\mathrm{st}_{\overline{\sigma}}(b_k)$ in terms of the functor $\Phi$, we simply apply the functors $\D$ and $\Phi$ to \eqref{eq:DivContractionDiv1}:
\begin{equation}\label{eq:CommutativeDiagramPhi}
\begin{gathered}
\xymatrix{
\cO_S\otimes K \ar[r]\ar[d]^{\mathrm{ev}} & \cO_S\otimes U \ar[r]\ar[d]^{\mathrm{ev}} & \cO_S\otimes Q\ar[d]^{\mathrm{ev}}\\
\cO_S\otimes W\ar[r]\ar[d] & \D(F_k) \ar[r]\ar[d] & \D(F)\ar[d]\\
\Phi(\cO_S\otimes W^\vee [1])\ar[r] & \Phi(F_k) \ar[r] & \Phi(F)
}
\end{gathered}
\end{equation}
for the graded vector spaces

\[
\begin{aligned}
K:=&\mathbf{R}\Hom(\cO_S,\cO_S)\otimes W,\\ U:=&\mathbf{R}\Hom(\cO_S,\D(F_k)),\\  Q:=&\mathbf{R}\Hom(\cO_S,\D(F)).
\end{aligned}
\]
By using the isomorphism $U\cong U_{2k}$, we can identify the first line of \eqref{eq:CommutativeDiagramPhi} with
\[
\cO_S\otimes \left(W \oplus W[-2]\right) \longrightarrow \cO_S\otimes U_{2k} \longrightarrow \cO_S\otimes \left(U_{2k}/W \oplus W[-1]\right),
\]
and so we deduce the exact triangle
\[
\cO_S\otimes \left(U_{2k}/W\right) \to \D(F) \to \Phi(F_k),
\]
as we wanted. Notice that this is nothing but the Mukai involution in divisibility~1, studied in~\cite{Bea:MukaiInvolution,Kieran:Involutions}.
\end{proof}

\begin{rema}\label{rmk:Section2appliesDiv1}
Observe that the moduli space $M$, the involution $\tau$, and the  contraction $\phi\colon M\rightarrow \overline{M}$ satisfy the requirements of Section~\ref{sec:surgery}:  \ref{enum:surgery1} is Proposition~\ref{azione}, \ref{enum:surgery2} and \ref{enum:surgery3} are by construction,
and \ref{enum:surgery5} follows from Lemma~\ref{lem:InvolutionDiv1}. Hence in the divisibility 1 case, the specialization we consider for the proof of the Main Theorem is the pair $(\overline{M}, \overline{\lambda})$, where $\overline{\lambda}$ denotes the ample Cartier divisor class induced by $\lambda$ on $\overline{M}$, equipped with the involution $\overline{\tau}$.
\end{rema}

\subsection{Divisibility 2}\label{subsec:div2}

Let $(S,h)$ be a polarized K3 surface of genus $g$ with $\NS(S)=\Z \cdot h$. Throughout the present subsection we assume that the genus is divisible by~4. Let 
$v = \left(0,h,1-g\right)$ and let $M:=M_h(v)$.

Let
\begin{equation*}
f = \vartheta\left(0,0,-1\right),\quad 
\delta= \vartheta\left(2,-h,\frac{g}{2}\right)=\cl(\Delta),\quad 
\lambda =\vartheta \left(2,-h,\frac{g}{2}-1\right)=f+\delta,
\end{equation*}
where $\Delta$ is the divisor defined in Proposition~\ref{prop:Deltadiv2}. Notice that $\lambda$ has square $2$ and divisibility $2$.
Recall from Section~\ref{subsec:examples} that the class $f$ induces the Lagrangian fibration
\[
\pi\colon M \to |\cO_S(H)|\cong \P^g,
\]
and is thus one of the two rays of the nef cone.

In order to describe the wall and chamber decomposition for the moduli spaces $M$, we introduce two integers $c,d\in\Z$ with $c\geq0$, $d\geq-1$, and satisfying the following two conditions:
\begin{enumerate}[(a)]
\item\label{enum:div2Nef1} $4 d + (2c+1)^2 \leq g-1$
\item\label{enum:div2Nef2} $\frac{(g-1)c^2-d}{2c+1} \in \Z$.
\end{enumerate}
Then we set
\[
\mu_{c,d} := \frac{g-1-4d-(2c+1)^2}{2(2c+1)^2} \geq 0
\]
and
\begin{align*}
& a_{c,d} := \left(2c+1, -c\cdot h, \frac{(g-1)c^2-d}{2c+1} \right)\\
& \widetilde{a}_{c,d}:=\lambda + \mu_{c,d} f.
\end{align*}
Notice that by definition $d=a_{c,d}^2/2$, $(\vartheta^{-1}(\widetilde{a}_{c,d}),a_{c,d})=0$, and by~\ref{enum:div2Nef2} $a_{c,d} \in H^*_\mathrm{alg}(S,\Z)$. We define an ordering on the pairs $(c,d)$ as above by using the slope $\mu_{c,d}$:
\[
(c,d)\preceq(c',d')\quad \text{ if and only if }\quad \mu_{c,d}\leq \mu_{c',d'}.
\]
With respect to this ordering, the first values in descending order are:
\[
(0,-1) \succeq (0,0) \succeq \dots
\]

The Mori chamber decomposition of the movable cone is more interesting than in the divisibility~1 case discussed in the previous subsection.

\begin{lemm}\label{lem:NefDiv2}
The movable cone for $M$ is
\[
\Mov(M) = \R_{\geq0} f+ \R_{\geq0} \lambda.
\]
The ordered rays generated by $\widetilde{a}_{c,d}$, where $c,d\in\Z$, $c\geq0$, $d\geq-1$, satisfy \ref{enum:div2Nef1} and \ref{enum:div2Nef2}, give the chamber decomposition of $\Mov(M)$.
\end{lemm}

\begin{proof}
This is an immediate calculation from Theorem~\ref{thm:walls}. Indeed, we look for Mukai vectors $a\in H_\mathrm{alg}^*(S,\Z)$ satisfying $a^2\geq-2$, $0\leq (a,v)\leq g-1$.
There are only two possibilities for $(a,v)$: either $(a,v)=0$ (and so, $a^2=-2$), or $(a,v)=g-1$.
The first case gives the ray generated by $\lambda$.
For the second case, if we write $a=(r,-c\cdot h,s)$, then the condition $(a,v)=g-1$ gives $r=2c+1$ and, by setting $d:=a^2/2$, we can directly express $s$ as a function of $c$ and $d$; hence, these $a$ are the vectors $a_{c,d}$, whose corresponding ray in the positive cone $\Pos(M)$ is exactly the one generated by $\widetilde{a}_{c,d}$. The fact that this is inside $\Mov(M)$ corresponds to the positivity of its slope $\mu_{c,d}$, which is guaranteed by \ref{enum:div2Nef1}.
\end{proof}

The birational maps $f_{c,d}$ corresponding to the vectors $a_{c,d}$ in Lemma~\ref{lem:NefDiv2} can be explicitly described and have a very simple form: they are all Mukai flops over products of moduli spaces of stable objects.

\begin{lemm}\label{lem:MukaiFlops}
For $c,d\in\Z$ as above, the divisor class $\widetilde{a}_{c,d}$ induces a Mukai flop whose exceptional locus is the projective bundle, in the analytic topology,
\[
P_{c,d}:=\PP_{Z_{c,d}}(\mathcal{V}_{c,d})
\]
where
\[
Z_{c,d} = M_{\sigma_{c,d}}(a_{c,d}) \times M_{\sigma_{c,d}}(v-a_{c,d})
\]
is a product of smooth HK manifolds of dimension $2d+2$, where $\mathcal{V}_{c,d}$ is the twisted vector bundle of rank $(a_{c,d},v-a_{c,d})$ induced by the relative Ext-sheaf, and where $\sigma_{c,d}\in \Stab^{\mathrm{geom}}(\Db(S))\subset\Stab^\dagger(\Db(S))$ is a stability condition associated to $\widetilde{a}_{c,d}$.
\end{lemm}

\begin{proof}
We use Example~\ref{ex:rk2}.
Indeed $v=(0,h,1-g)$ satisfies all the assumptions with $x=2$ and $y=-1$.
By \cite[Lemma 9.2]{BM:projectivity}, the associated divisor class $\ell_{\sigma_{\alpha,-1/2}}$ on $M$ satisfies 
\begin{align*}
&\ell_{\sigma_{\alpha,-1/2}} \to f,\qquad \text{for } \alpha \to \infty\\
&\ell_{\sigma_{\alpha,-1/2}} \to \lambda,\qquad \text{for } \alpha \to \frac{1}{2\sqrt{g-1}},
\end{align*}
up to multiplicative constants.
Hence, all walls in the movable cone will come from a stability condition in this segment $\sigma_{\alpha,-1/2}$, for $\alpha > \frac{1}{2\sqrt{g-1}}$, thus finishing the proof.
\end{proof}

\begin{rema}\label{rmk:TotallySemistAndExplicitModels}
The above proof allows us to say a bit more about birational models of $M$ and the structure of stable objects.

(a) We first notice that there are no totally semistable walls (in the sense of \cite[Definition 2.20]{BM:walls}) for the Mukai vector $v$ on the segment $\sigma_{\alpha,-1/2}$, $\alpha>\frac{1}{2\sqrt{g-1}}$; concretely, this means that the general object $E$ in $M$ is $\sigma_{\alpha,-1/2}$-stable, for all $\alpha>\frac{1}{2\sqrt{g-1}}$.
Indeed, by using \cite[Theorem 5.7]{BM:walls} we see that in our case such a totally semistable wall exists if and only if there exist $\alpha_0>\frac{1}{2\sqrt{g-1}}$ and a $\sigma_{\alpha_0,-1/2}$-stable spherical object $R$ in $\coh^{-1/2}(S)$ which satisfies $(v(R),v)<0$ and is a Jordan-H\"older factor for a general element $E$ in $M$ on $\sigma_{\alpha_0,-1/2}$.
Example~\ref{ex:rk2} then forces $(v(R),v)=-1$, which is impossible, being a multiple of $1-g$ with $4\,|\,g$.

(b) Given $c,d\in \ZZ$ as above, we let
\[
\alpha_{(c,d)} > \frac{1}{2\sqrt{g-1}}
\]
be such that the divisor class $\ell_{\sigma_{\alpha_{(c,d)},-1/2}}$ is equal to the class $\widetilde{a}_{c,d}$, up to a positive constant (namely, we can choose $\sigma_{c,d}$ as $\sigma_{\alpha_{(c,d)},-1/2}$).
We define for $0<\epsilon\ll 1$:
\begin{equation*}
    \begin{split}
        &M_{c,d}:= M_{\sigma_{\alpha_{(c,d)}+\epsilon,-1/2}}(v),\\
        &M_{c,d}':= M_{\sigma_{\alpha_{(c,d)}-\epsilon,-1/2}}(v),\\
        &M_{\mathrm{last}}:=M_{\sigma_{\frac{1}{2\sqrt{g-1}}+\epsilon,-1/2}}(v).
    \end{split}
\end{equation*}
Notice that $M_{\mathrm{last}}$ is equal to $M'_{c_0,d_0}$ where $(c_0,d_0)$ is the minimum $(c,d)$ for the total ordering that we have defined above. Moreover
\[
M_{0,-1}=M \quad \text{ and } \quad M_{0,-1}'=M_{0,0},
\]
$\widetilde{a}_{c,d}$ is nef on $M_{c,d}$, and $\lambda$ is nef on $M_{\mathrm{last}}$.

For $E\in M_{c,d}$, we have that $E$ is in the projective bundle $P_{c,d}$ if and only if its Jordan-H\"older filtration with respect to $\sigma_{\alpha_{(c,d)},-1/2}$ has only two terms of the form
\begin{equation}\label{eq:MukaiflopJH}
R \to E \to R'
\end{equation}
with $R\in M_{\sigma_{\alpha_{(c,d)},-1/2}}(a_{c,d})$ and $R'\in M_{\sigma_{\alpha_{(c,d)},-1/2}}(v-a_{c,d})$.
The twisted vector bundle $\mathcal{V}_{c,d}$ then has fiber $\Ext^1(R',R)$ over the point $(R,R')$.
\end{rema}

\begin{exam}\label{ex:FanoCptFirstFlop}
We can describe the first flop, induced by $\widetilde{a}_{0,-1}$ as follows.
The exceptional locus of the flop consists of those $E=\mathcal{O}_C$, for $C\in|\cO_S(H)|$, and thus can be identified with the image of the zero section $\check{\PP}^g:=\PP(S,H^0(\cO_S(H)))$.
The Jordan-H\"older filtration with respect to $\sigma_{\alpha_{(0,-1)},-1/2}$ is given by
\[
\cO_S \to E \to \cO_S(-H)[1].
\]
On the other side of the flop, the objects are given as non-trivial extensions of the form
\[
\cO_S(-H)[1]\to E' \to \cO_S,
\]
which are then parametrized by $\PP(S,H^2(\cO_S(-H)))=\PP(H^0(S,\cO_S(H))^\vee)=\PP^g$.
\end{exam}

As in the divisibility~1 case, we have an explicit description of the divisorial contraction induced by $\lambda$.
Let $A$ be the (unique) spherical stable vector bundle on $S$ with Mukai vector $v(A)=(2,-h,g/2)$, as in~\eqref{eq:eccoA}.
For $c,d\in\ZZ$ as above and $k\ge 1$, we let
\begin{equation}\label{deltacidi}
\begin{split}
\Delta_{c,d} & := \left\{ F \in M_{c,d}\, :\, \hom_S(A,F)>0 \right\},\\
\Delta_{c,d}(k) & := \left\{ F \in M_{c,d}\, :\, \hom_S(A,F)=k \right\},
\end{split}
\end{equation}
and similarly on $M_{\mathrm{last}}$
\begin{equation*}
\begin{split}
\Delta_{\mathrm{last}} & := \left\{ F \in M_{\mathrm{last}}\, :\, \hom_S(A,F)>0 \right\},\\
\Delta_{\mathrm{last}}(k) & := \left\{ F \in M_\mathrm{last}\, :\, \hom_S(A,F)=k \right\}.
\end{split}
\end{equation*}
Notice that $\Delta_{c,d}$, respectively $\Delta_{\mathrm{last}}$, is the strict transform of $\Delta$ on $M_{c,d}$, respectively $M_{\mathrm{last}}$.

An argument analogous to the one given for Lemma~\ref{lem:DivContractionDiv1} and Lemma~\ref{lemm:Delstratdiv1} gives the following two results.

\begin{lemm}\label{lem:DivContractionDiv2}
The class $\lambda$ induces a divisorial contraction
\[
\phi\colon M_\mathrm{last} \to \overline{M}\]
with  exceptional divisor $\Delta_\mathrm{last}$. 
\end{lemm}

\begin{lemm}\label{lem:Delstratdiv2}
The divisor $\Delta_\mathrm{last}$ is prime and 
\begin{equation}\label{stratdel2}
\Delta_\mathrm{last}=\bigsqcup\limits_{1\le k\le\lfloor \sqrt{g} \rfloor}\Delta_{\mathrm{last}}(k),\qquad \ov{\Delta_{\mathrm{last}}(i)}=\bigsqcup\limits_{i\le k\le\lfloor \sqrt{g} \rfloor}\Delta_{\mathrm{last}}(k).
\end{equation}
Moreover let $\overline{\sigma}\in \Stab^{\mathrm{geom}}(\Db(S)) \subset\Stab^\dagger(\Db(S))$ be a stability condition such that $\ell_{\overline{\sigma}}=\lambda$. Then each $\Delta_{\mathrm{last}}(k)$  is isomorphic to a locally trivial (in the analytic topology) Grassmannian bundle $\Gr(k,\cU_{2k})$ over the $(2g-2k^2)$--dimensional moduli space $M^\mathrm{st}_{\overline{\sigma}}(b_k)$ of \emph{stable} objects, where
\[
b_k = - \left(2k, (1-k) h, \left(1-\frac k2\right) g -1\right)
\]
and
$\cU_{2k}$ is a twisted vector bundle of rank $2k$ induced by the relative Ext-sheaf.
\end{lemm}

\begin{proof}
We prove that $\Delta_\mathrm{last}(k)$ is a Grassmannian bundle as stated.  Explicitly, choose $\overline{\sigma}=\sigma_{\alpha,\beta}$ with $(\alpha,\beta)$ sufficiently close to $(1/2\sqrt{g-1},0)$ and $\beta<-1/2$. Then $F\in \Delta_{\mathrm{last}}(k)$ if and only if its Jordan--H\"older filtration with respect to $\overline{\sigma}$ is
\[
A\otimes W^\vee \to F \to F_k,
\]
where $F_k\in M^\mathrm{st}_{\overline{\sigma}}(b_k)$, $\dim W=k$, $U_{2k}:=\Hom(F_k,A[1])$, and the twisted vector bundle $\cU_{2k}$ has fiber equal to $U_{2k}$ over the point $F_k\in M^\mathrm{st}_{\overline{\sigma}}(b_k)$. By irreducibilty of moduli spaces of stable objects on $K3$ surfaces it follows that $\Delta_\mathrm{last}$  is irreducible. Since its class $\delta$ is indivisible $\Delta_\mathrm{last}$  is a prime divisor. The equalities in~\eqref{stratdel2} follow as in the proof of Lemma~\ref{lemm:Delstratdiv1}.
\end{proof}

Recall from Section \ref{subsec:examples} that the moduli space $M$ is equipped with an involution $\tau$ induced by the functor
\[
\Psi := \cO_S(-H)\otimes\D\colon \Db(S)^{\mathrm{op}}\xlongrightarrow{\cong}\Db(S)
\]
and that $\tau$ acts fiberwise with respect to the Lagrangian fibration $\pi$.
We now verify that $\tau$ is moreover compatible with the Mukai flops $f_{c,d}$ and divisorial contraction $\phi$ introduced above. 

\vspace{0.5cm}

\begin{lemm}\label{lem:InvolutionDiv2}
We adopt the notation of Lemma~\ref{lem:MukaiFlops} and Lemma~\ref{lem:DivContractionDiv2}.
\begin{enumerate}[(i)]
\item\label{enum:InvolutionDiv21} The involution $\tau$ on $M$ is compatible with the Mukai flops $f_{c,d}$. More precisely, the functor $\Psi$ induces an isomorphism $M_{\sigma_{c,d}}(a_{c,d}) \xrightarrow{\cong} M_{\sigma_{c,d}}(v-a_{c,d})$ so that the induced action on $M_{\sigma_{c,d}}(a_{c,d}) \times M_{\sigma_{c,d}}(v-a_{c,d})$ exchanges the two factors and given a fixed point $(R_{c,d},\Psi(R_{c,d}))$ the induced actions on the fibers $\P(\mathcal{V}_{c,d |_{(R_{c,d},\Psi(R_{c,d}))}})$ and $\P(\mathcal{V}^\vee_{c,d |_{(R_{c,d},\Psi(R_{c,d}))}})$ are naturally dual to each other.
\item\label{enum:InvolutionDiv22} The involution $\tau$ on $M$ maps each $\Delta_{\mathrm{last}}(k)$ to itself and is compatible with the Grassmannian bundle $\Delta_{\mathrm{last}}(k)\to M^\mathrm{st}_{\overline{\sigma}}(b_k)$; the corresponding involution on $M^\mathrm{st}_{\overline{\sigma}}(b_k)$ is induced by the anti-autoequivalence $\Phi\colon\Db(S)^{\mathrm{op}}\xrightarrow{\cong}\Db(S)$ given by $\Phi:=\mathrm{ST}_{A}\circ(\cO_S(-H)\otimes\D)$.
In particular, $\tau$ induces a regular involution $\ov{\tau}\colon\ov{M}\xrightarrow{\cong}\ov{M}$. 
\end{enumerate}
\end{lemm}

\begin{proof}
Part~\ref{enum:InvolutionDiv21} follows immediately from Lemma~\ref{lem:MukaiFlops} and the following observations.
Let $E$ be an object in $M_{c,d}$ and consider its Jordan-H\"older filtration with respect to $\sigma_{c,d}=\sigma_{\alpha_{(c,d)},-1/2}$ given in \eqref{eq:MukaiflopJH}:
\[
R_{c,d} \to E \to R_{c,d}'
\]
By applying the functor $\Psi$, we get a triangle:
\[
\Psi(R_{c,d}') \to \Psi(E) \to \Psi(R_{c,d})
\]

By using Example~\ref{ex:rk1}, $\Psi(R_{c,d})$ (respectively $\Psi(R_{c,d}')$) is also $\sigma_{c,d}$-stable with Mukai vector $v-a_{c,d}$ (respectively $a_{c,d}$).
Hence, the involution $\tau$ is compatible with the Mukai flop and it induces via the functor $\Psi$ an isomorphism  $M_{\sigma_{c,d}}(a_{c,d}) \xrightarrow{\cong} M_{\sigma_{c,d}}(v-a_{c,d})$.
The identification of $\tau$ with $\Psi$ thus shows that the action of $\tau$ on  the product $M_{\sigma_{c,d}}(a_{c,d}) \times M_{\sigma_{c,d}}(v-a_{c,d})$ exchanges the two factors.

It follows that there is an induced natural action on the universal relative Ext-(twisted) sheaf $\mathcal V_{c,d}$ on the product $M_{\sigma_{c,d}}(a_{c,d}) \times M_{\sigma_{c,d}}(v-a_{c,d})$ and thus on the projective bundle $P_{c,d}=\PP_{Z_{c,d}}(\mathcal{V}_{c,d})$. The induced action on the fibers $\P(\mathcal{V}_{c,d |_{(R_{c,d},\Psi(R_{c,d}))}})$ and $\P(\mathcal{V}^\vee_{c,d |_{(R_{c,d},\Psi(R_{c,d}))}})$ are then indeed naturally dual to each other.

For part~\ref{enum:InvolutionDiv22}, the argument is analogous to the proof of Lemma~\ref{lem:InvolutionDiv1}. 
\end{proof}

\vspace{1cm}

\begin{exam}\label{ex:Div2Flopsg48}
In low genera the walls and Mukai flops can be easily described.

(a) $g=4$. In such a case we must have $c=0$ and the only vectors allowed are $a_{0,-1}=(1,0,1)$ and $a_{0,0}=(1,0,0)$.
The birational transformations are given by
\[
\xymatrix{
M=M_{0,-1} \ar[d]^{\pi} \ar[rd]  \ar@{<-->}[rr]^{f_{0,-1}} && M_{0,0} \ar[ld] \ar[rd]  \ar@{<-->}[rr]^{f_{0,0}} && M_{\mathrm{last}} \ar[ld] \ar[d]^{\phi}\\
|\cO_S(H)|\cong\P^4 & \overline{M}_{0,-1} && \overline{M}_{0,0} & \overline{M}.
}
\]
where $f_{0,-1}$ is the Mukai flop at the $0$-section of $\pi$ described in Example~\ref{ex:FanoCptFirstFlop} and $f_{0,0}$ is the Mukai flop at the $\P^2$-bundle in the Zariski topology over the product $S\times S$.
The destabilizing exact sequences are
\[
\mathcal{O}_S \to E \to \mathcal{O}_S(-H)[1]
\]
for $f_{0,-1}$ (as already observed), and
\[
\mathscr{I}_p \to E \to \D(\mathscr{I}_q)\otimes\mathcal{O}_S(-H)
\]
for $f_{0,0}$, where $p,q\in S$.
The induced action of $\tau$ is trivial on $\P(\Hom(\cO_S(-H),\cO_S))$ and it is of type $(1,2)$ on $\P(\Ext^1(\D(\mathscr{I}_p)\otimes\mathcal{O}_S(-H),\mathscr{I}_p))$. 

The above example  gives an explicit decomposition of the birational map between the LLSvS variety of a cubic fourfold with a node and $M_h(0,h,-3)$, see the discussion at the end of Subsection~\ref{subsec:pazzaidea}. 

(b) $g=8$. Also in this case we must have $c=0$ and the vectors allowed are $a_{0,-1}=(1,0,1)$, $a_{0,0}=(1,0,0)$, and $a_{0,1}=(1,0,-1)$.
The birational transformations are
\[
\xymatrix{
M=M_{0,-1} \ar[d]^{\pi} \ar[rd]  \ar@{<-->}[rr]^{f_{0,-1}} && M_{0,0} \ar[ld] \ar[rd]  \ar@{<-->}[rr]^{f_{0,0}} && M_{0,1} \ar[ld] \ar[rd]  \ar@{<-->}[rr]^{f_{0,1}} && M_{\mathrm{last}} \ar[ld] \ar[d]^{\phi}\\
|\cO_S(H)|\cong\P^8 & \overline{M}_{0,-1} && \overline{M}_{0,0} && \overline{M}_{0,1} & \overline{M}.
}
\]
where $f_{0,-1}$ is the Mukai flop at the $0$-section, $f_{0,0}$ the Mukai flop at the $\P^6$-bundle in the Zariski topology over the product $S\times S$, and $f_{0,1}$ the Mukai flop at the $\P^4$-bundle in the Zariski topology over the product $S^{[2]}\times S^{[2]}$. Destabilizing sequences for $f_{0,-1}$ and $f_{0,0}$ are analogous to the $g=4$ case, while the one for $f_{0,1}$ is
\[
\mathscr{I}_{\Gamma} \to E \to \D(\mathscr{I}_{\Gamma'})\otimes\mathcal{O}_S(-H)
\]
where $\Gamma,\Gamma'\subset S$ are $0$-dimensional closed subschemes of length 2.
The induced action of the involution $\tau$ is trivial on $\P(\Hom(\cO_S(-H),\cO_S))$, it is of type $(1,6)$ on $\P(\Ext^1(\D(\mathscr{I}_p)\otimes\mathcal{O}_S(-H),\mathscr{I}_p))$, and of type $(2,3)$ on $\P(\Ext^1(\D(\mathscr{I}_{\Gamma})\otimes\mathcal{O}_S(-H),\mathscr{I}_{\Gamma}))$.

(c) The first higher rank flop ($c>0$) starts at $g=12$, with the vector $a_{1,-1}=(3,-h,4)$: the stable factors in the Jordan-H\"older filtration have ranks $\pm 3$, and the geometric intuition is less clear.
\end{exam}

\begin{rema}
As in the divisibility~1 case, the moduli space $M$, the involution $\tau$, and the contraction $\phi\colon M\rightarrow \overline{M}$ satisfy the requirements of Section~\ref{sec:surgery} via Proposition~\ref{azione}, Lemma~\ref{lem:DivContractionDiv2} and Lemma~\ref{lem:InvolutionDiv2}.
Hence, analogously to the divisibility~1 case, in the divisibility~2 case, the specialization we consider for the proof of the Main Theorem is the pair $(\overline{M}, \overline{\lambda})$, where $\overline{\lambda}$ denotes the ample Cartier divisor class induced by $\lambda$ on $\overline{M}$, equipped with the involution $\overline{\tau}$. 
\end{rema}


\section{Connected components of the fixed locus}\label{sec:ConnectedComponents}

The goal of this section is to study the connected components of the fixed locus of the involution $\tau$ on the Lagrangian fibrations introduced in Section~\ref{sec:birational}, for a general polarized K3 surface.
This study is easier on these models since the fixed locus here  corresponds either to (the closure of the) locus of $2$--torsion points or of theta-characteristics on the linear sections of the K3 (see Section~\ref{subsec:Trigonal}).
For a general K3 surface, we compute the monodromy group and then rule out the existence of exotic components fully supported over singular curves. This allows us to compute the number of connected components.

As an immediate application, in Section~\ref{subsec:ProofThm1} we use the results in Section~\ref{sec:surgery} to prove the Main Theorem in the divisibility~1 case.
The proof of the Main Theorem in the divisibility~2 case requires a more in-depth analysis of how the connected components of the fixed locus of $\tau$ behave under wall-crossing. This will be done in Section~\ref{sec:ExplicitMMP}.

\subsection{Components of the fixed locus on Lagrangian fibrations}\label{subsec:Trigonal}

Let $(S,h)$ be a polarized K3 surface of genus $g$ for which every curve in the linear system $ |\cO_S(H)|$ is integral.
We consider the two moduli spaces of Section~\ref{subsec:examples}
\[
M_h(0,h,1-g) \qquad \text{ and } \qquad M_h(0,h,0) 
\]
together with the antisymplectic involutions $\cO_S(-H)\otimes\D$ on $M_h(0,h,1-g)$ and $\D$ on $M_h(0,h,0)$, which by abuse of notation are both denoted by $\tau$.

We recall that since $\tau$ is antisymplectic, the fixed locus $\Fix(\tau)$ on each of these moduli spaces is a smooth Lagrangian submanifold.
The two key results of this section are the following propositions.

\begin{prop}\label{tuttobene}
The irreducible components of $\Fix(\tau)$ on $M_h(0,h,1-g)$ are
\begin{equation}\label{simom}
\Fix(\tau)=\Sigma\sqcup \Omega,
\end{equation}
where $\Sigma$ is the image of the zero section (and thus parametrizes the structure sheaves of the curves in $ |\cO_S(H)|$) and where $ \Omega$  is the closure of the locus parametrizing sheaves $i_{C,*}(\xi)$ where  $C\in |\cO_S(H)|$ is a smooth curve and $\xi$ is a non-trivial square root of $\cO_C$ (and thus is the closure of $2$-torsion points).
\end{prop}

\begin{prop}\label{effetto}
The irreducible components of $\Fix(\tau)$ on $M_h(0,h,0)$ are
\begin{equation}\label{piuomeno}
\Fix(\tau)={\mathsf S}^{+}\sqcup{\mathsf S}^{-},
\end{equation}
where ${\mathsf S}^{+}$, respectively ${\mathsf S}^{-}$, is the closure of the locus parametrizing sheaves $i_{C,*}(\xi)$ where $C\in |\cO_S(H)|$ a smooth curve and $\xi$ is an even, respectively odd, theta-characteristic.
\end{prop}

\begin{rema}\label{rmk:NonIntegral}
If $|\cO_S(H)|$ contains non integral curves then $\Fix(\tau)$ may very well have more than two irreducible components; this is possible because $M_h(v)$ is singular in such cases.
For example, if $(S,h)$ is hyperelliptic covering a rational scroll $\FF_0$ or $\FF_1$ (depending on the parity of $g$) then $\Fix(\tau)\subset M_h(0,h,1-g)$ has $\lfloor\frac{g+3}{2}\rfloor$  irreducible components dominating  $|\cO_S(H)|$. In fact this follows from the explicit description of theta-characteristics on a hyperelliptic curve: the irreducible components dominating  $|\cO_S(H)|$ are indexed by the $l$ (including $l=-1$) appearing in the first two displayed equations on p.~191 of~\cite{mumtheta}. 
\end{rema}

Before proving Proposition~\ref{tuttobene} and Proposition~\ref{effetto} we need some preliminary results.
To start with, we look at the monodromy representation with integral coefficients for smooth curves on a general polarized K3 surface.

For a polarized K3 surface $(S,h)$, we denote by $|\cO_S(H)|^0\subset |\cO_S(H)|$ the open subset parametrizing smooth curves.
Notice that $|\cO_S(H)|^0$ is dense unless $(S,h)$ is unigonal (we recall that $(S,h)$ is unigonal if $H\equiv R+gE$ where $R$ is a smooth rational curve, $E$ is an elliptic curve and $R\cdot E=1$).

\begin{prop}\label{montrans}
Let $(S,h)$ be a general polarized K3 surface of genus $g\ge 3$.
Let $C_0\in |\cO_S(H)|^0$ and consider $H^1(C_0,\ZZ)$ with the intersection product.
Then the natural representation
\[
\pi_1( |\cO_S(H)|^0,C_0)\lra  \Sp(H^1(C_0,\ZZ))
\]
is surjective. 
\end{prop}

\begin{proof}
According to Theorem 3 in~\cite{beaumon} (based on~\cite{janssen1,janssen2}), and the discussion thereafter, it suffices to exhibit a polarized K3 surface $(S,h)$ such that
\begin{enumerate}[(i)]
\item\label{enum:monodromy1} there exists $C\in|\cO_S(H)|$ with a singularity of Type $E_6$, and 
\item\label{enum:monodromy2} there exists  $(C'+C'')\in|\cO_S(H)|$ such that  $C',C''$ are reduced, they intersect transversely, and $C'\cdot C''$ is odd.
\end{enumerate}
Trigonal K3 surfaces (we recall that a polarized $K3$ surface $(S,h)$ is trigonal if $S$ has an elliptic pencil $|E|$  such that $H\cdot E=3$, and hence all the curves in $|H|$ are trigonal) provide  such examples. We get trigonal polarized K3's of  genus $g$ by one of the following constructions:
\begin{enumerate}[(a)]
\item\label{enum:trigonal1} $g\equiv 0\pmod{3}$. Let $S\subset\PP^3$ be a smooth quartic surface containing a plane cubic curve $E$. Let $D\in|\cO_{S}(1)|$. Then $H:=D+\frac{(g-3)}{3}E$ is an ample divisor with $H\cdot H=2g-2$, and $(S,h)$ is trigonal because $H\cdot E=3$.
\item\label{enum:trigonal2} $g\equiv 1\pmod{3}$. Let $S\subset\PP^4$ be the transverse intersection of a quadric cone $Q$ and  cubic hypersurface. Let $D\in|\cO_{S}(1)|$. The projection from the vertex of $Q$ defines a finite map $S\to \PP^1\times\PP^1$. Composing with one of the two projections $\PP^1\times\PP^1\to\PP^1$ we get an elliptic fibration $S\to\PP^1$; let $E$ be an element of the fibration. Then $H:=D+\frac{g-4}{3}E$ is an ample divisor with $H\cdot H=2g-2$, and $(S,h)$ is trigonal because $H\cdot E=3$.
\item\label{enum:trigonal3} $g\equiv 2\pmod{3}$, $g\ge 5$: Let $S\subset\PP^2\times\PP^1$ be a smooth element of $|\cO_{\PP^2}(3)\boxtimes\cO_{\PP^1}(2)|$. Let $p_i$ be the projections of $S$ to $\PP^2$ and $\PP^1$ respectively, and let $D\in|p_1^{*}\cO_{\PP^2}(1)|$, $E\in|p_2^{*}\cO_{\PP^2}(1)|$. Then $H:=D+\frac{g-2}{3}E$ is an ample divisor with 
$H\cdot H=2g-2$, and $(S,h)$ is trigonal because $H\cdot E=3$.
\end{enumerate}

For any such trigonal $(S,h)$ as above Item~\ref{enum:monodromy2} holds. In fact we can take $C'\in|E|$ and $C''\in |H-E|$ (which has no base locus).
We will show that in each of~\ref{enum:trigonal1}, \ref{enum:trigonal2}, \ref{enum:trigonal3}, one can choose $(S,h)$ such that Item~\ref{enum:monodromy1} holds as well.
In each of~\ref{enum:trigonal1}, \ref{enum:trigonal2}, $H\equiv D+mE$ with $m\ge 0$. Since $E$ has no base locus, it suffices to show that there exists a choice of $(S,h)$ with a $D\in|\cO_{S}(1)|$ with a singularity of Type $E_6$.
In~\ref{enum:trigonal3}, $H\equiv D+mE$ with $m\ge 1$: hence it suffices to show that there exists a choice of $(S,h)$ with a $C\in|\cO_{S}(D+E)|$ with a singularity of Type $E_6$. 

In case~\ref{enum:trigonal1}, we let $[x,y,z,t]$ be homogeneous coordinates on $\PP^3$ and  $S$ be 
\[
x^4+y^3t-z^4+zt^3=0.
\]
As is easily checked $S$ is smooth, the curve $S\cap V(z)$ has an $E_6$ singularity at $[0,0,0,1]$. Moreover $S$ contains the line $R:=V(t,x-z)$, and hence it contains the plane cubic cuves in the pencil $|H-R|$. 

In case~\ref{enum:trigonal2}, let $[x_0,\ldots,x_5]$ be homogeneous coordinates on $\PP^4$ and let $S$ be
\[
x_0x_1-x_2 x_3=x_4^3+f(x_0,\ldots,x_3)=0.
\]
The vertex of the quadric cone containing $S$ is $p=[0,\ldots,0,1]$. Projection from $p$ defines a finite map $S\to T$, where $T\subset\PP^3$ is the smooth quadric $x_0x_1-x_2 x_3=0$.
If the curve   $B\subset\PP^3$ defined by $x_0x_1-x_2 x_3=f(x_0,\ldots,x_3)=0$ is smooth, then $S$ is smooth.
If $\Lambda\subset\PP^3$ is a plane such that there exists a point $q\in \Lambda\cap B$ with $\mult_q(\Lambda\cdot B)=4$, then the hyperplane $\la p,\Lambda\ra\subset\PP^4$ intersects $S$ in a point of Type $E_6$. One easily writes down such examples.

In case~\ref{enum:trigonal3}, first we define a curve $C\subset\PP^1\times\PP^2$ with an $E_6$ singularity, and then we argue that there exists a smooth $S$ as above containing $C$ as a divisor in  $|\cO_{S}(D+E)|$.
Let  $[x,y,z]$, $[s,t]$  be homogeneous coordinates on $\PP^2$ and $\PP^1$ respectively, and let $C$ be defined by
\[
x^3st+y^3s^2+y^3t^2=sx-tz=0.
\]
The curve $C$ has an $E_6$ singularity at $([1,0],[0,0,1])$.
Moreover  let $S\subset\PP^2\times\PP^1$ be defined by
\[
x^3st+y^3s^2+y^3t^2+\varphi(x,y,z;s,t)(sx-tz)=0,
\]
where $\varphi(x,y,z;s,t)$ is a bi-homogeneous polynomial of degree $3$ in $x,y,z$ and degree $2$ in $s,t$.
Then $C\subset S$ and $C\in|\cO_{S}(D+E)|$.
It remains to prove that there exists $\varphi(x,y,z;s,t)$ such that $S$ is smooth. This follows from Bertini's theorem.
\end{proof}

As another preliminary result, we need to exclude the existence of connected components in the fixed locus entirely contained over the locus parametrizing singular curves.

\begin{prop}\label{alldom}
Let  $(S,h)$ be a polarized K3 surface of genus $g$ such that every curve in the linear system $ |\cO_S(H)|$ is integral.
Then every irreducible component of $\Fix(\tau)$ dominates $|\cO_S(H)|$.
\end{prop}

\begin{proof}
Over a smooth curve $C$, the fiber $\pi^{-1}(C)\cap\Fix(\tau)$ of the induced morphism $\Fix(\tau) \to |\cO_S(H)|$ is finite.
Since every component of $\Fix(\tau)$ is of dimension $g$, it is enough to give a dimension estimate of the fibers over the locus in $ |\cO_S(H)|$ parametrizing singular curves.
Since every  curve in the linear system $ |\cO_S(H)|$ is integral, the two Lagrangian fibrations for $v=(0,h,0)$ and $v=(0,h,1-g)$ are fiberwise isomorphic and one can furthermore choose the isomorphisms to be compatible with the involutions on the two moduli spaces.
It is therefore enough to prove the statement for the case $v=(0,h,1-g)$.

We claim that if $C$ is a curve of cogenus\footnote{The \emph{cogenus} of a reduced curve is the difference between the arithmetic genus and geometric genus.} $\delta>0$ (thus $C$ is singular), then $\dim \pi^{-1}(C)\cap\Fix(\tau) < \delta$.
Since by \cite[Proposition 2.4.2]{dCRS}, the locally closed subset of $ |\cO_S(H)|$ parametrizing curves with cogenus $\delta$ has codimension $\ge \delta$, the proposition follows from the claim.

To prove the claim we use some well-known facts about compactified Jacobian of integral locally planar curves. Standard references are \cite{Rego,Cook-thesis,Beauville-counting}.
Let $C \in |\cO_S(H)|$ be as above, i.e.,  an integral curve of arithmetic genus $g$ and cogenus $\delta >0$.
Since $C$ has locally planar singularities, the degree zero compactified Jacobian $\pi^{-1}(C)=\jb(C)$, parametrizing torsion free sheaves of rank $1$ and degree $0$ on $C$, is irreducible of dimension $g$ and contains $\Pic^0(C)$, the locus parametrizing locally free sheaves, as a dense open subset.
Since $\Fix(\tau) \cap \Pic^0(C)$ is $0$-dimensional, we only have to worry about non locally free sheaves. 
The complement of $\Pic^0(C)$ in $\jb(C)$ is non empty because $C$ is singular and it can be described in the following way.
For every non locally free sheaf $\cF\in\jb(C)$ there exists a unique partial normalization $n'\colon C' \to C$ such that  $n'_* \cO_{C'}=\cE nd(\cF)$ and a unique torsion free sheaf $\cF'$ on $C'$ of rank $1$ such that $\cF=n'_* \cF'$ (see~\cite[Section 2]{Beauville-counting}).
The degree $0$ Picard variety $\Pic^0(C)$ acts on $\jb(C)$ by tensorization and for  $\cF=n'_* \cF'$ as above, $\cL \otimes  \cF \cong \cF$ if and only if $ \cL \in \ker[{n'}^*: \Pic^0(C) \to \Pic^0(C')]$ (see~\cite[Lemma 2.1]{Beauville-counting}).
The involution $\tau$ commutes with the action of $\Pic^0(C)$.

The possible dimensions of the stabilizers of the $\Pic^0(C)$ action on $\jb(C)$ go from  $0$ to $\delta$.
For every  $0\le \delta' \le\delta$, let $Z_{\delta'} \subset \jb(C)$ be the locally closed subset of points whose stabilizer has dimension equal to $ \delta'$.
This defines a finite $\tau$-invariant and $\Pic^0(C)$-invariant stratification on $\jb(C)$ with $Z_0=\Pic^0(C)$.
Now set $g'=g-\delta'$, so that $\dim Z_{\delta'}=g'+\epsilon_{g'}$ for some $\epsilon_{g'} \ge 0$. 
For $\delta'>0$, $Z_{\delta'}$ is strictly contained in $\jb(C)$, so
\begin{equation} \label{epsilon}
g'+\epsilon_{g'}< g \quad \text{and} \quad \epsilon_{g'}< \delta'\le\delta.
\end{equation}

We want to show that
\[
\codim_{Z_{\delta'}} (Z_{\delta'} \cap \Fix(\tau)) \ge g'.
\]
By \eqref{epsilon}, this implies that $\dim Z_{\delta'} \cap \Fix(\tau) \le \epsilon_{g'}<\delta$ and hence the claim.

For every partial normalization $n'\colon C'\to C$, set $g':=g(C')$ and $\delta':=g-g'$.
Then $0<\delta' \le \delta$. For any integer $d$, let $\jb^{d}(C')$ denote the moduli space of degree $d$ torsion free sheaves of rank $1$ on $C'$.
The pushforward map $n'_*: \jb^{-\delta'}(C') \to \jb(C)$ is a closed embedding (see~\cite[Lemma 3.1]{Beauville-counting}) and the image is clearly preserved by the action of $\Pic^0(C)$. In particular there is an induced action of $\Pic^0(C')$ on $n'_*(\jb^{-\delta'}(C'))$, and this action is free on the locus, denoted by $n'_*(\jb^{-\delta'}(C')^\circ)$, of sheaves $\cF$ with $n'_* \cO_{C'}=\cE nd(\cF)$. In particular,  $n'_*(\jb^{-\delta'}(C')^\circ) \subset Z_{\delta'}$.
By relative duality,
\begin{equation} \label{finiteduality}
\cH om_C(n'_*\cF', \cO_C) \cong n_* \cH om_{C'}(\cF', \omega_{n'}),
\end{equation}
where $\omega_{n'}$ is the relative dualizing sheaf of the finite morphism $n'\colon C' \to C$ (see~\cite[Section 3]{Hulek-Laza-Sacca}).
The involution $\tau$ acts on $\jb(C)$ preserving $n'_*(\jb^{-\delta'}(C')^\circ)$ and its $\Pic^0(C')$-orbits.
Consider a sheaf $\cF=n'_* \cF'$ in $n'_*(\jb^{-\delta'}(C')^\circ)$ and  let $\cG$ be another  sheaf in the same $\Pic^0(C')$-orbit.
Write $\cG=n'_* \cG'$ with $\cG'=\cF' \otimes \cL'$ and $\cL' \in \Pic^0(C')$.
Suppose both $\cF$ and $\cG$ are $\tau$-invariant.
Then by  \eqref{finiteduality} and the fact that $n'_*$ is an embedding, $\cL'^{\otimes 2}=\cO_{C'}$. The conclusion is that for each $\Pic^0(C')$-orbit in $n'_*(\jb^{-\delta'}(C')^\circ)$ its intersection with $\Fix(\tau)$ is a torsor over the $2$-torsion line bundles on $C'$ and is thus zero dimensional.
It follows that every irreducible component of the intersection of $\Fix(\tau)$ with $Z_{\delta'}$ has codimension at least $g'$. 
\end{proof}

\begin{proof}[Proof of Proposition~\ref{tuttobene}]
By Proposition~\ref{alldom} the equality in~\eqref{simom} holds.
Since $\Sigma$ is isomorphic to $|\cO_C(H)|\cong \PP^g$ it is irreducible. 

It remains to prove that  $\Omega$ is  irreducible.
Let $\Omega^0:=\Omega \cap \pi^{-1}(|\cO_S(H)|^0)$. The restriction of $\pi$ defines a map
\[
\pi^0\colon\Omega^0\lra |\cO_S(H)|^0
\]
which is proper and a local homeomorphism. It suffices to prove that $\Omega^0$ is connected.  
If $g\ge 3$, we can deform and assume that $(S,h)$ is general, and then this follows at once from Proposition~\ref{montrans}. 

If $g=2$ we argue as follows. 
The map  $f\colon S\to  |\cO_S(H)|^{\vee}\cong\PP^2$ is a double cover ramified over a smooth sextic curve.
Let $D\subset S$ be the ramification divisor.
Let $C_0\in |\cO_S(H)|^0$ be a smooth curve. Then $C_0$ intersects transversely $D$ in $6$ points $x_1,\ldots,x_6$, and
\[
\Pic^0(C_0)[2]\setminus\{0\}=\{x_i-x_j\}_{1\le i< j\le 6}.
\]
Hence it suffices to prove that the  action of $\pi_1(|\cO_S(H)|^0,C_0)$ on the set $\{x_1,\ldots,x_6\}$ is the full symmetric group.
This holds by the Uniform Position Theorem applied to the hyperplane (i.e., line) sections of the branch curve $f(D)\subset\PP^2$.
\end{proof}

\begin{proof}[Proof of Proposition~\ref{effetto}]
By Proposition~\ref{alldom} the equality in~\eqref{piuomeno} holds.
It remains to prove that  ${\mathsf S}^{\pm}$ is  irreducible.
Let  $({\mathsf S}^{\pm})^0:={\mathsf S}^{\pm}\cap \pi^{-1}(|\cO_S(H)|^0)$. The restriction of $\pi$ defines  maps
\[
(\pi^{\pm})^{0}\colon({\mathsf S}^{\pm})^0\lra |\cO_S(H)|^0.
\]
which are proper and local homeomorphisms.
It suffices to prove that $({\mathsf S}^{\pm})^0$ is connected.
If $g\ge 3$ this follows  from Proposition~\ref{montrans} and the theory of theta-characteristics, see for example p.~294 in~\cite{acgh}.

If $g=2$ one argues as in the analogous case in the proof of Proposition~\ref{tuttobene}.
\end{proof}

Recall that an irreducible closed subset $Z$ of a variety $X$ is called \emph{constant cycle} if all points of $Z$ (or all points of a dense open subset of $Z$) are rationally equivalent in $X$.
By \cite[Lemma 3.2]{Lin}, if $X$ is a HK manifold then this is equivalent to the cycle map $\CH_0(Z)_\QQ \to \CH_0(X)_\QQ $ being constant.

\begin{prop}\label{prop:constantcycle}
Let $M$ be  $M_h(0,h,1-g)$, respectively $M_h(0,h,0)$.
The irreducible subvarieties  $\Sigma, \Omega \subset M$, respectively $\mathsf{S}^-, \mathsf{S}^+ \subset M$, are constant cycle subvarieties.
In particular, they have no non trivial holomorphic forms.
\end{prop}

\begin{proof}
The second statement follows from the first using the Mumford-Roitman Theorem \cite{Mumford:Chow,Roi:0cycles}, \cite[Theorem 10.17] {Voisin:book2}.
The first statement follows from arguments in \cite{Lin}, which we now recall for the reader's sake.
Let $\cH$ be an ample $\tau$-invariant line bundle on $M$. For a general fiber $M_t$ of $\pi$, take for the origin one of the fixed points of $\tau$.
The restriction $\cH_t$  is a symmetric line bundle with respect to this group structure on the abelian variety $M_t$ and the points of $\Fix(\tau) \cap M_t$ are precisely the two torsion points of $M_t$.
By~\cite[Lemma~3.4]{Lin}, for any $2$-torsion point $x \in M_t$, the class $D[x] $ is rationally equivalent to $\cH_t^g$, where $D=\deg(\cH_t^g)$.
Since all fibers $M_t$ are rationally equivalent in $M$, it follows that all points of an irreducible component of $\Fix(\tau)$ are rationally equivalent in $M$.
\end{proof}

\subsection{Proof of the Main Theorem, divisibility~1}\label{subsec:ProofThm1}

As an immediate application of the results in the previous section, we can prove the Main Theorem, in the easier case of divisibility~1.
We just need two preliminary observations.

\begin{prop}\label{prop:SminusContained}
Let $(S,h)$ be a polarized K3 surface of genus $g$ with $\NS(S)=\ZZ\cdot h$.
We keep the notation of Sections~\ref{subsec:div1} and~\ref{subsec:Trigonal}.
\begin{enumerate}[(i)]
\item \label{enum:Thetaodd} The irreducible subvariety $\mathsf{S}^-$ is contained in $\Delta$, it is uniruled, and the ruling is induced by the ruling of $\Delta$.
\item \label{enum:Thetaeven} The irreducible subvariety $\mathsf{S}^+$ is not contained in $\Delta$.
\end{enumerate}
\end{prop}

\begin{proof}
The first statement of Part~\ref{enum:Thetaodd} is clear by the definitions of $\mathsf{S}^-$ and $\Delta$; see Remark~\ref{delta=theta}.
The second statement is \cite[Proposition 2.2]{Farkas-Verra}: let $\eta$ be an odd theta characteristic on a smooth curve $C$ corresponding to a general point in $\mathsf{S}^-$.
Then $h^0(C,\eta)=1$ and $\eta=\mathcal{O}_C(D)$ for a unique effective divisor of degree $g-1$ on $C$.
The member of the ruling of $\Delta$ corresponding to the pencil of curves in $|\mathcal{O}_S(H)|$ containing $D$ is then clearly contained in $\mathsf{S}^-$. 
The statement in Part~\ref{enum:Thetaeven} follows from the well known fact that on every smooth curve there exist non-effective theta characteristics (see~\cite[Theorem 1.10(ii)]{Harris-theta} and use the fact that for the general curve there are no effective even theta characteristics).
\end{proof}

\begin{coro}\label{cor:Essepmbar}
Let $\overline{\mathsf{S}}^\pm$ be the image of $\mathsf{S}^\pm$ under the divisorial contraction $\varphi: M \to \overline{M}$.
\begin{enumerate}[(i)]
\item The general fiber of the induced morphism $\varphi: \mathsf{S}^- \to \overline {\mathsf{S}}^-$ has dimension $1$ and thus $\dim \overline{\mathsf{S}}^-= g-1$.
\item The induced morphism $\varphi: \mathsf{S}^+ \to \overline{\mathsf{S}}^+$ is birational.
\end{enumerate}
\end{coro}

\begin{proof}[Proof of the Main Theorem, divisibility 1 case.]
We keep notation as above, so that $\overline M$ is the singular model of $M$ under the divisorial contraction of Lemma~\ref{lem:DivContractionDiv1}, $\tau$ is the involution on $M$ of Proposition~\ref{azione} and $\overline\tau$ is the induced involution on $\overline M$.

As observed in Remark~\ref{rmk:Section2appliesDiv1}, we can apply the results in Section~\ref{sec:surgery}.
Let $\mathcal M \to D$ be the pullback of the universal family over $\Def(\overline M)$ to a general disc  $D \subset \Def(\overline M, \overline \tau)$. 
By Proposition~\ref{uscire} and Corollary~\ref{unaclasse}, $\mathcal M \to D$ is a smoothing of $\overline M$ and there is a global involution $\mathfrak{T}$ on $\mathcal M$ which preserves the fibers $\mathcal M_t$ and is such that, for $t \neq 0$, $\mathfrak{T}_t$ is the involution whose fixed locus we are interested in and for $t=0$, $\mathfrak{T}_0=\overline \tau$. 
Let $\Fix(\mathfrak{T}, \mathcal M)$ be the fixed locus of $\mathfrak{T}$ on $\mathcal{M}$. The fiber at zero satisfies $\Fix(\mathfrak{T}, \mathcal M)_0=\Fix(\overline \tau, \overline M)$ which, by Corollary \ref{cor:Essepmbar}, is the disjoint union $\overline {\mathsf{S}}^+ \coprod \overline {\mathsf{S}}^-$ of a $g$-dimensional component and a $(g-1)$-dimensional component.
The $(g-1)$-dimensional component $ \overline {\mathsf{S}}^-$ is of course contained in the singular locus of $\overline M$.

At a general point of  $\overline {\mathsf{S}}^+$, a tangent space computation shows that  $\Fix(\mathfrak{T}, \mathcal M)$ is smooth of dimension $g+1$ and that the morphism $\Fix(\mathfrak{T}, \mathcal M) \to D$ is smooth of relative dimension $g$.
Let $\mathcal M^*:=\mathcal M \setminus \mathcal M_0$ and let $\mathcal F:= \overline{\Fix(\mathfrak{T}, \mathcal M^*})$ be the closure of the fixed locus on $\mathcal M^*$.
Then $\mathcal F\to D$ is a flat surjective morphism, which is generically smooth along the central fiber $\mathcal F_0=\overline {\mathsf{S}}^+$, and so $\mathcal F_t$ is connected for any $t \in D$.
\end{proof}


\section{Proof of the Main Theorem in divisibility~2}\label{sec:ExplicitMMP}

In this section we complete the proof of the Main Theorem in the divisibility~2 case.
The approach is similar to the case of divisibility~1, with two extra issues. First, we describe how the fixed locus of $\tau$ in the Lagrangian fibration behaves under the flopping transformations of Section~\ref{subsec:div2} and we verify that no connected component of this fixed locus is created or destroyed in the process.
This, which is the content of Section~\ref{subsec:FlopInvolutions}, immediately gives that there are at most two connected components in the singular model $\overline{M}$.
The second problem is to show that there are exactly two components. This is done in Section~\ref{subsec:linearization}.
To deal with this, we notice that the line bundle $L_{\mathrm{last}}$ on $M_{\mathrm{last}}$ with $\mathrm{c}_1(L_{\mathrm{last}})=\lambda$ has a natural linearization with respect to the action of the involution $\tau$.
Then we study the action of $\tau$ on the fibers of $L_{\mathrm{last}}$ over the fixed locus: the two connected components correspond to different irreducible representations.
This will show that even in the singular model $\overline{M}$ the two connected components are disjoint and, since they both intersect the smooth locus, they each are limits of a component of the fixed locus we are interested in.

We start with a general section on flops and involutions on HK manifolds, Section~\ref{subsec:FlopInvolutionsGeneral}, where we outline the geometric picture of what happens to the fixed loci when we pass through a Mukai flop.

\subsection{Flops and involutions, general results}\label{subsec:FlopInvolutionsGeneral}

The goal of this section is to describe locally the birational transformation induced on the fixed loci by a Mukai flop.

We start with elementary considerations.
Let $V$ be a vector space and let $\tau$ be an involution on $\PP V$.
We can linearize the action to lift the involution to $V$ and get dual actions on $V^\vee$ and $\PP V^\vee$.
By abuse of notation we denote all of these actions by $\tau$.
The dual actions on $V$ and $V^\vee$ are such that the $(\pm 1)$-eigenspaces $V^\pm$ and $(V^\vee)^\pm$ satisfy $(V^\vee)^\mp=\mathrm{Ann}(V^\pm)$.
There is a natural action of $\tau$ on the Euler sequence
\[
0 \to \Omega_{\PP V}^1 \to \cO_{\PP V}(-1) \otimes V^\vee \to \cO_{\PP V} \to0,
\]
and the $(-1)$-eigenbundle satisfies
\begin{equation}\label{eq:ActionNormalBundle}
(\Omega^1_{\PP V})^-=(\cO_{\PP V}(-1) \otimes V^\vee)^{-}.
\end{equation}
Notice that while the action on $V$ depends on the choice of the linearization, the action on $\cO_{\PP V}(-1) \otimes V^\vee$ does not.
A similar statement also holds true in the relative setting, when we consider a relative action on a projective bundle over a base.

We now consider the following setting.
Let $\hat M$ be an open (in the analytic or Zariski topology) subset of a HK manifold and $\hat P \subset \hat M$ be a coisotropic subvariety which is isomorphic to the projectivization of a vector bundle over a smooth manifold $\hat Z$.
To fix notation,  we let $\mathcal V$ be the vector bundle on $\hat Z$ such that $\hat P \cong \PP_{\hat Z} \mathcal V\xrightarrow{\rho}\hat Z$.
The holomorphic symplectic form then induces  an isomorphism $N_{\hat P/\hat M}\cong \Omega_{\hat P/\hat Z}$ and the restriction of the symplectic form to $\hat P$ is the pullback of a holomorphic symplectic form on $\hat Z$. 
Assume that there is an antisymplectic involution $\tau$ acting on $\hat M$ and preserving the projective bundle $\hat P \to \hat Z$, with induced antisymplectic involution on $\hat Z$ denoted by $\tau_{\hat Z}$.
Let $\psi\colon \hat M \dashrightarrow \hat{M}'$ be the Mukai flop of $\hat M$ along $\hat P$ and denote by $\hat P'$ the exceptional locus of $\psi^{-1}$, so that $\hat P'=\PP_{\hat Z} \mathcal V^\vee\xrightarrow{\rho'}\hat Z$, where $\mathcal V^\vee$ is the dual bundle.
With these assumptions, there is an induced regular antisymplectic involution $\tau'$ on $\hat{M}'$, restricting to the dual action on $\hat P' $.

We want to describe the general behavior of the restriction of $\psi$ to the connected components of $\Fix(\tau, \hat M)$.
We first examine the case where a connected component is contained in $\hat P$.

\begin{prop}\label{prop:fixedindet}
Let $\hat{M}$ be an open subset of a HK manifold and $\hat{P}\subset \hat{M}$ a coisotropic subvariety such that $\hat{P}
\cong \mathbb{P}_{\hat{Z}}\mathcal{V}$ for some vector bundle 
$\mathcal{V}$ over a smooth manifold $ \hat{Z}$. If there is an antisymplectic involution $\tau$ acting on $\hat{M}$ such that $\tau$ preserves the projective bundle $\rho'\colon \hat P\rightarrow \hat Z$ and a component $\hat{F}$ of the fixed locus $\Fix(\tau,\hat M)$ is contained in $\hat{P}$, then the following hold:
\begin{enumerate}[(i)]
\item\label{enum:fixedindet1} There exists a component ${\hat Y}\subset{\hat Z}$ of the fixed locus of the induced antisymplectic involution $\tau_{\hat Z}$ on $\hat Z$ such that ${\hat F}=\rho^{-1}({\hat Y})$.
\item\label{enum:fixedindet2} The Mukai flop $\psi\colon \hat M \dashrightarrow \hat{M}'$ of $\hat M$ along $\hat P$ inducing the regular antisymplectic involution $\tau'$ on $\hat{M}'$ restricting to the dual action on 
$\hat P'=\PP_{\hat Z} \mathcal V^\vee\xrightarrow{\rho'}\hat Z$ has the property that 
the subset $(\rho')^{-1}({\hat Y})\subset {\hat P}'$ is a component  of the fixed locus $\Fix(\tau',\hat M')$.
\end{enumerate}
\end{prop}

\begin{proof}
Part~\ref{enum:fixedindet1} holds by an easy dimension count.
By~\ref{enum:fixedindet1}, the action of $\tau$ on each fiber of $\hat P \to \hat Z$ over a point of $\hat Y$ is trivial.
Since $\rho'\colon {\hat P}'\to {\hat Z}$ is the dual of $\hat P \to \hat Z$ and the action of $\tau'$ is the dual action, it follows that the action of $\tau'$ on each fiber of $\hat P' \to \hat Z$ over a point of $\hat Y$ is trivial.
This proves~\ref{enum:fixedindet2}.
\end{proof}

To complete the general picture, we record here what happens in the case when a connected component is not contained in $\hat P$.

\begin{prop}\label{prop:localmodelflips}
Under the hypotheses of Proposition \ref{prop:fixedindet}, let $\hat F$ be a component of $\Fix(\tau, \hat M)$ such that $\hat F\not\subset \hat P$ and let ${\hat F}'$ be the proper transform of $\hat F$ in $\hat{M}'$.
Then $\hat{F}'$ is a connected component of $\Fix(\tau',\hat{M}')$ and the Mukai flop $\psi$ restricts to a birational map $\phi\colon \hat F \dashrightarrow \hat F'$, which is a finite series of disjoint flips, one for each connected component $\hat \Gamma$ of $\hat F\cap \hat P$.
\end{prop}

More precisely, we will show that any connected component $\hat \Gamma$ of $\hat F\cap \hat P$ (with its reduced induced scheme structure) is identified with a projective bundle
\begin{equation}\label{eq:LocalModelFlips}
p\colon \hat \Gamma\cong \PP_{\hat W} \mathcal{V}_{\hat{\Gamma}} \longrightarrow \hat W,
\end{equation}
where $\hat W \subset \hat Z$ is a connected component of $\Fix(\tau_{\hat Z}, \hat Z)$ and $\mathcal V_{\hat{\Gamma}}$ is one of the two eigenbundles of $\mathcal V_{|{\hat W}}$.
The normal bundle of $\hat{\Gamma}$ in $\hat P$ satisfies $N_{\hat{\Gamma}/\hat P} \cong \cO_p(-1) \otimes \mathcal V'_{\hat{\Gamma}}$, where $\mathcal V'_{\hat{\Gamma}}:= Ann(\mathcal V_{\hat{\Gamma}}) \subset \mathcal V^\vee_{|\hat{\Gamma}}$, and locally around $\hat{\Gamma}$ the birational map $\phi$ is the corresponding standard flip of $\hat{F}$ along $\hat{\Gamma}$, namely $\hat{\Gamma}$ is replaced by $\hat{\Gamma}':=\PP_{\hat{W}} \mathcal V'_{\hat{\Gamma}}$.

\begin{proof}
If $\hat{\Gamma}$ is a connected component of $\hat{F}\cap \hat{P}$, then $\hat{\Gamma}$ is a connected component of the fixed locus of $\tau$ on $\hat P$. Thus $\hat{\Gamma}$ is smooth and coincides with a projective bundle $\PP_{\hat{W}} \mathcal V_{\hat{\Gamma}}$, as in~\eqref{eq:LocalModelFlips}.
Let $\tilde F$ be the proper transform of $\hat{F}$ in the blow up of $\hat M$ along $\hat P$.
Then $\tilde F$ is a component of the fixed locus of an involution and thus is smooth.
Hence the birational map $\tilde F \to \hat F$, which has connected fibers, has exceptional locus of codimension one and is isomorphic to a projective bundle over $\hat \Gamma$.
One can see that $\tilde F \to \hat F$ is, in fact, just the blowup of $\hat F$ along $\hat \Gamma$: by the universal property of blowing up, there exist a birational morphism $\tilde F \to \Bl_{\hat \Gamma} \hat F$, which has connected fibers and is finite; it is thus an isomorphism because $\Bl_{\hat \Gamma} \hat F$ is smooth.

To compute the normal bundle of $\hat \Gamma$ in $\hat F$ we look at the natural action of $\tau$ (and $\tau_{\hat Z}$) on several conormal/cotangent bundles sequences and at the corresponding $(\pm 1)$-eigenbundles. 
Comparing the first cotangent bundle sequence for $\hat \Gamma$ in $\hat F$ with that of $\hat \Gamma$ in $\hat M$ and using the fact that action on $\Omega_{\hat \Gamma}$ is trivial and the action on $N_{\hat F/\hat M}^\vee$ is $(-1)$, we see that $N_{\hat \Gamma/\hat F}=(N_{\hat \Gamma/\hat M})^+$.
Similarly, looking at normal bundle sequence for $\hat \Gamma \subset \hat P \subset \hat M$ and using the fact that the action on $N_{\hat \Gamma/\hat P}$ is $(-1)$, we see that $(N_{\hat \Gamma/\hat M})^+=(N_{\hat P/\hat M})^+$.
Since $\tau$ is antisymplectic, the isomorphism $N_{\hat P/\hat M} \cong \Omega_{\hat P/ \hat Z}$ switches the $(\pm 1)$-eigenbundles.
By~\eqref{eq:ActionNormalBundle}, the claim on the normal bundle follows and hence also the one on the structure of $\phi$.
\end{proof}

\subsection{Flops and involutions in a special case}\label{subsec:FlopInvolutions}

In this section, we return to our more specific situation.
We let $(S,h)$ be a polarized K3 surface of genus $g$ such that $4\,|\,g$ and $\mathrm{NS}(S)=\ZZ\cdot h$. We are in the context of Section~\ref{subsec:div2} and keep the notation therein: $M:=M_h(v)$, its birational models are denoted $M_{c,d}$, $M_{\mathrm{last}}$ the divisorial contraction associated to $\lambda$ is denoted by $\phi\colon M_{\mathrm{last}}\to\overline{M}$, and the involution is denoted by $\tau$.
We also recall the two fixed components $\Sigma(\cong\PP^g)$ and $\Omega$ of $\tau$ in $M$ introduced in Proposition~\ref{tuttobene}.
The goal is to show the following result.

\begin{prop}\label{prop:FlopDoesNotIncreaseNumberComponents}
The fixed locus of $\tau$ on $M_{\mathrm{last}}$ has exactly two connected components:
\[
\Fix (\tau, M_{\mathrm{last}}) = \Sigma_{\mathrm{last}} \sqcup \Omega_{\mathrm{last}}.
\]
Moreover, $\Sigma_{\mathrm{last}}$ is rational while $\Omega_{\mathrm{last}}$ is birational to $\Omega$ and has no non trivial holomorphic forms.
\end{prop}

We adopt the notation of Remark~\ref{rmk:TotallySemistAndExplicitModels}. In particular $f_{c,d}\colon M_{c,d}\dra  M'_{c,d}$ is
the Mukai flop induced by the divisor class $\widetilde{a}_{c,d}$. Recall the definition of $\Delta_{c,d}$ in~\eqref{deltacidi}.

\begin{lemm}\label{lem:EremainsStableIfHom0}
For all $(c,d)$ the indeterminacy locus of the flop $f_{c,d}\colon M_{c,d}\dra M'_{c,d}$ is contained in $\Delta_{c,d}$.
\end{lemm}

\begin{proof}
It suffices to prove that if $E\in M_{c,d}$ and $\Hom_S(A,E)=0$, 
then $E$ is $\sigma_{\alpha,-1/2}$-stable for all $\alpha_{(c,d)}\geq\alpha>\frac{1}{2\sqrt{g-1}}$.
If $E$ is not $\sigma_{\alpha,-1/2}$-stable for some $\alpha_{(c,d)}\geq \alpha>\frac{1}{2\sqrt{g-1}}$, since by Remark~\ref{rmk:TotallySemistAndExplicitModels}(a) there are no totally semistable walls on the segment $\sigma_{\alpha,-1/2}$, then we can assume that there exists $\alpha_{(c',d')}<\alpha_{(c,d)}$ such that $E$ is not $\sigma_{\alpha_{(c',d')},-1/2}$-stable.
Then, we have a Jordan-H\"older filtration with respect to $\sigma_{\alpha_{(c',d')},-1/2}$ of the form
\[
R_{c',d'} \to E \to R_{c',d'}'
\]
with $v(R_{c',d'})=a_{c',d'}$ and $v(R_{c',d'}')=v-a_{c',d'}$.

By applying the functor $\Hom_S(A,\blank)$, we get an exact sequence:
\[
\Hom_S(A,R_{c',d'}'[-1]) \to \Hom_S(A,R_{c',d'}) \to \Hom_S(A,E).
\]
Since $R_{c',d'}'\in\coh^{-1/2}(S)$ is $\sigma_{\alpha_{(c',d')},-1/2}$-stable and the slope of $A[1]$ is infinite, we have
\[
\hom_S(A,R_{c',d'}'[-1])=\hom_S(A[1],R_{c',d'}')=0.
\]
Similarly, by Serre duality, since $R_{c',d'},A[1]\in\coh^{-1/2}(S)$, we have
\[
\hom_S(A,R_{c',d'}[2])=\hom(R_{c',d'},A)=\hom(R_{c',d'},A[1][-1])=0.
\]
An immediate calculation, by using condition \ref{enum:div2Nef1} in Section~\ref{subsec:div2}, shows that
\[
(\delta,\vartheta(a_{c',d'}))\leq -(2c'+1)<0,
\]
which gives then
\[
0\neq \Hom_S(A,R_{c',d'}) \hookrightarrow \Hom_S(A,E),
\]
a contradiction.
\end{proof}

In $M=M_{0,-1}$ the fixed locus $\Fix(\tau,M)$ has two connected components, namely $\Sigma$ and $\Omega$. 
Motivated by Lemma~\ref{lem:EremainsStableIfHom0} we determine whether $\Sigma$ or $\Omega$ is contained in $\Delta=\Delta_{0,-1}$. First we deal with $\Sigma$. 
We need a few preliminary results on the spherical vector bundle $A$.

\begin{lemm}\label{lem:RestrictionAtoCurves}
Let $C\in|\cO_S(H)|$ be a curve.
Then the restriction map
\[
H^0(S,A^\vee) \longrightarrow H^0(C,A^\vee_{|C})
\]
is an isomorphsim and
\[
h^0(S,A^\vee)=h^0(C,A^\vee_{|C})=2+\frac{g}{2}.
\]
\end{lemm}

\begin{proof}
By the explicit description in~\eqref{eq:eccoA}, we deduce immediately that $h^0(S,A)=h^1(S,A)=0$ and, by Serre duality, $h^0(S,A^\vee)=g/2+2$ and $h^1(S,A^\vee)=h^2(S,A^\vee)=0$.

Finally, from the exact sequence
\[
0 \to A^\vee(-H)\cong A \to A^\vee \to A^\vee_{|C} \to 0,
\]
for $C\in |\cO_S(H)|$, we deduce our statement.
\end{proof}

Let  $r:=h^0(S,A^\vee)/2=1+g/4$.
Then $r$ is a positive integer because $4\,|\,g$.

\begin{lemm}\label{lem:generalPointsSectionAvanish}
Let $x_1,\dots,x_r\in S$ be general distinct points.
Then
\[
H^0(S,\mathcal{I}_{\{x_1,\dots,x_r\}}\otimes A^\vee)=0.
\]
\end{lemm}

\begin{proof}
Since $A^\vee$ is globally generated of rank~2, if $x_1\in S$ then
\[
h^0(S,\mathcal{I}_{x_1}\otimes A^\vee)=h^0(S,A^\vee)-2=\frac g2.
\]
Since $A^\vee$ is slope stable with $\mathrm{c}_1(A)=h$ and $\NS(S)=\ZZ\cdot h$, if $F\subset A^\vee$ is a rank~1 subsheaf then $\mathrm{c}_1(F)=mh$ with $m\leq0$. Thus, if $F$ is globally generated then $m=0$, so $F\cong\cO_S$ and $h^0(S,F)=1$.
Hence, since $g/2>0$ is even, the subsheaf of $A^\vee$ generated by $H^0(S,\mathcal{I}_{x_1}\otimes A^\vee)$ must have rank~2.
It follows that if $x_2\in S$ is general, then
\[
h^0(S,\mathcal{I}_{\{x_1,x_2\}}\otimes A^\vee)=h^0(S,A^\vee)-4=\frac g2 -2.
\]
Iterating, we get the result.
\end{proof}

\begin{lemm}\label{lem:MulteplicityFormula}
In $M$, we have
$\Sigma\subset \Delta$ and 
\[
\mult_{\Sigma}(\Delta)=h^0(C,A^\vee_{|C})=2+\frac{g}{2}.
\]
\end{lemm}

\begin{proof}
It suffices to prove that if $C\in|\cO_S(H)|$ is a smooth curve
then
\[
\mathrm{mult}_{[\cO_C]}\,\Delta|_{\Pic^0(C)} = h^0(C,A_{|C}^\vee).
\]
Let $r=h^0(S,A^\vee)/2$ as above and let $x_1,\dots,x_r\in S$ be general points.
By a dimension count, there exists a smooth curve $C\in |\cO_S(H)|$ containing $x_1,\dots,x_r$.
By Lemma~\ref{lem:generalPointsSectionAvanish}, we have that $H^0(S,\mathcal{I}_{\{x_1,\dots,x_r\}}\otimes A^\vee)=0$.
Since by Lemma~\ref{lem:RestrictionAtoCurves} the restriction map $H^0(S,A^\vee)\to H^0(C,A^\vee_{|C})$ is an isomorphism, we deduce that
\[
H^0(C,A^\vee_{|C}(-(x_1+\dots+x_r))=0.
\]
The statement follows now from \cite[Lemma 6.1]{CMTB:Singularities} (see also \cite{CMF:CubicThreefolds}).
\end{proof}

Next we deal with $\Omega$. 
\begin{lemm}\label{lem:NoHomsOmega}
In $M$, we have $\Omega \not\subset \Delta$. 
\end{lemm}

\begin{proof}
If $(V,\Theta)$ is a principally polarized abelian variety of dimension~$g$ and we let
\[
\eta\colon V \to |\cO_V(2\Theta)|^\vee \cong \PP^{2^g-1}
\]
be the morphism  associated to the divisor $2\Theta$, then the points $\eta(x)$, for $x\in V[2]$, span the whole $\PP^{2^g-1}$. This follows from the fact  that $H^0(V,\cO_V(2\Theta))$ is an irreducible representation of the theta group  $\cG(\cO_V(2\Theta))$ (see~\cite{mum-eq-abvars}, Sect.~1, Thm.~2) and the following argument. We have the exact sequence
\[
1 \lra \CC^{*} \lra \cG(\cO_V(2\Theta)) \lra 
K(\cO_V(2\Theta)) \lra 1
\]
where $K(\cO_V(2\Theta))\cong V[2]$ and the center $\CC^{*}$ acts on $H^0(V,\cO_V(2\Theta))$ via the \lq\lq identity\rq\rq\ character (multiplyed by $\Id$). Hence the  representation of the theta group induces a representation of $V[2]$ on 
$|\cO_V(2\Theta)|$ and hence also on $|\cO_V(2\Theta)|^\vee$.
  In the latter representation $\eta(V)$ is mapped to itself and the action corresponds to the  action by translations. Hence the span of $\eta(V[2])$ is invariant for the $V[2]$ action. Since 
 $H^0(V,\cO_V(2\Theta))$ is an irreducible representation of the theta group  $\cG(\cO_V(2\Theta))$,  so is the 
 dual representation $H^0(V,\cO_V(2\Theta))^{\vee}$, and it follows that the span of $\eta(V[2])$ is the whole of $|\cO_V(2\Theta)|^\vee$.
 
Now let  $C\in|\cO_S(H)|$ be a smooth curve. By Proposition~\ref{prop:Deltadiv2} the restriction of 
$\Delta=\Delta_{0,-1}$ to $\pi^{-1}(C)=\Pic^{0}(C)$ is a non zero section of twice the natural principal polarization. Since $\Delta|_{\Pic^{0}(C)}$ vanishes at $[\cO_C]\in \Pic^0(C)[2]$, it follows from the general result on abelian varieties that we have recalled above that there exists  $[\xi]\in \Pic^0(C)[2]$, such that $i_{C,*}(\xi)$ is not contained in $\Delta$.
\end{proof}

We notice that the irreducibility of $\Omega$ (namely, Proposition~\ref{tuttobene}) gives the slightly stronger result that, for general smooth $C\in|\cO_S(H)|$, $i_{C,*}(\xi)$ is not contained in $\Delta$, for all $\xi\neq\cO_C$ in $\Pic^0(C)[2]$.

We now focus our attention on the first flop $f_{0,-1}$, which is the flop of $\Sigma$, see Example~\ref{ex:FanoCptFirstFlop}.
To simplify notation in what follows, we set:
\[
M':=M_{0,0},\ f:=f_{0,-1}\colon M \dra M',\ \Delta':=\Delta_{0,0}.
\]
We moreover let $\Sigma'\subset M'$ be the ``flopped" $\Sigma$ (i.e., the indeterminacy locus of the inverse $f^{-1}$), and we let $\Omega':=f(\Omega)$. Clearly, $\Fix(\tau, M')=\Sigma'\sqcup \Omega'$.

\begin{lemm}\label{lem:NoHomsSigma}
In $M'$, we have $\Omega' \not\subset \Delta'$ and $\Sigma' \not\subset \Delta'$.
\end{lemm}

\begin{proof}
Since $\Omega$ is disjoint from the center of the flop $f$ we have $\Omega' \not\subset \Delta'$ by Lemma~\ref{lem:NoHomsOmega}.
It remains to prove that  $\Sigma' \not\subset \Delta'$.

Let
\begin{equation}\label{eq:firstflopdiagram}
\xymatrix{
& \Xi \ar@{^{(}->}[r] \ar@{->>}[dl] \ar@{.>}[rrrd] & \widetilde{M} \ar@{->>}[dr]^{g'} \ar@{->>}[dl]_{g} && \\
\Sigma \ar@{^{(}->}[r] & M\ar@{<-->}[rr]_{f}^{\cong} && M' & \Sigma'  \ar@{_{(}->}[l]
}
\end{equation}
be the resolution of the Mukai flop $f$, with common exceptional divisor $\Xi$. 
Let $d:=\mathrm{mult}_{\Sigma'}\Delta'$.
We only need to show that $d=0$.

Let us denote by $\widetilde{\Delta}$, respectively $\widetilde{\Delta}'$, the strict transform of $\Delta$, respectively $\Delta'$, on $\widetilde{M}$.
Then $\widetilde{\Delta}=\widetilde{\Delta}'$ and we have
\[
\begin{cases}
g^*\Delta = \widetilde{\Delta} + \left(\frac g2 + 2\right)\, \Xi\\
g^{\prime *}\Delta' = \widetilde{\Delta}' + d\, \Xi,
\end{cases}
\]
where we used that $\mathrm{mult}_{\Sigma}\Delta=g/2+2$, by Lemma~\ref{lem:MulteplicityFormula}.

Let $\widetilde{\gamma}\subset \Xi$ be a lift of a line $\gamma\subset\Sigma$. Then $g'(\widetilde{\gamma})=\mathrm{pt}\in M'$.
Thus, we have
\[
0=g^{\prime *}\Delta'\cdot\widetilde{\gamma} = \left(\widetilde{\Delta}' + d\, \Xi\right)\cdot\widetilde{\gamma} = \widetilde{\Delta}'\cdot\widetilde{\gamma} - d.
\]
Since $\widetilde{\Delta}'\cdot\widetilde{\gamma}=\widetilde{\Delta}\cdot\widetilde{\gamma}$, we get
\[
-\left(\frac g2 +2\right) = \Delta\cdot\gamma = g^*\Delta\cdot\widetilde{\gamma} = \widetilde{\Delta}\cdot\widetilde{\gamma} + \left(\frac g2 +2\right) \Xi\cdot\widetilde{\gamma} = d - \left(\frac g2 +2\right),
\]
and thus $d=0$, as we wanted.
\end{proof}

Concretely, as explained in Example~\ref{ex:FanoCptFirstFlop}, Lemma~\ref{lem:NoHomsSigma} means that for a general non-trivial extension
\[
\cO_S(-H)[1] \to E' \to \cO_S
\]
we have that $\Hom_S(A,E')=0$.

\begin{proof}[Proof of Proposition~\ref{prop:FlopDoesNotIncreaseNumberComponents}.]
Recall that the $(c,d)$ introduced in Subsection~\ref{subsec:div2} are totally ordered. The maximum is $(0,-1)$, and the minimum is ``{\rm last}'': thus    
$(0,-1)\succeq (0,0)\succeq\ldots \succeq {\rm last}$.   
We know that for $M'=M_{0,0}$, we have $\Fix(\tau, M')=\Sigma'\cup\Omega'$ and, by Lemma~\ref{lem:NoHomsSigma}, that neither $\Sigma'$ nor $\Omega'$ is contained in $\Delta'$. 
Let $(0,0)\succeq(c_0,d_0)$. 
By Lemma~\ref{lem:EremainsStableIfHom0} and Lemma~\ref{lem:NoHomsSigma} the rational map $M_{0,0}\dra M_{c,d}$ given by composing successive $f_{c,d}$ for $(0,0)\succeq(c,d)\succeq(c_0,d_0)$ is regular at the generic point of $\Sigma_{0,0}$ and of $\Omega_{0,0}$.
We let 
$\Sigma_{c_0,d_0},\Omega_{c_0,d_0}\subset M_{c_0,d_0}$ be the image of $\Sigma_{0,0},\Omega_{0,0}$ respectively.
Let us prove by descending induction on the $(c,d)'s$ that $\Fix(\tau,M_{c,d})=\Sigma_{c,d}\cup\Omega_{c,d}$.
Thus we assume the result for $M_{c,d}$ and we prove it for $M'_{c,d}$. Assume that it does not hold for $M'_{c,d}$.
Then there is a component of $\Fix(\tau,M'_{c,d})$ which is contained in the indeterminacy locus of $f_{c,d}^{-1}$.
The hypotheses of Proposition~\ref{prop:fixedindet} hold for ${\hat M}=M'_{c,d}$ and $\psi:=f_{c,d}^{-1}$. It follows that there is a component of $\Fix(\tau,M_{c,d})$ which is contained in the indeterminacy locus of $f_{c,d}$, contradicting the inductive hypothesis.
\end{proof}

As already remarked, as an immediate consequence of Proposition~\ref{prop:FlopDoesNotIncreaseNumberComponents} we can use the same argument as in the divisibility~1 case (in Section~\ref{subsec:ProofThm1}) to show that in the divisibility~2 case there are at most two connected components.

\subsection{Linearizations and fixed loci}\label{subsec:linearization}

To complete the proof of the Main Theorem, in the divisibility~2 case, we only have to show the following two results.
We define
\[
\begin{split}
    &\overline{\Sigma} := \phi(\Sigma_{\mathrm{last}})\\
    &\overline{\Omega} := \phi(\Omega_{\mathrm{last}}),
\end{split}
\]
where $\Sigma_{\mathrm{last}}, \Omega_{\mathrm{last}} \subset M_{\mathrm{last}} $ are the two connected components of the fixed locus of $\tau$ on $M_{\mathrm{last}}$. The fact that there are two follows from Proposition~\ref{prop:FlopDoesNotIncreaseNumberComponents}.

\begin{prop}\label{prop:FixedCmptsNotContainedInDelta}
Neither $\Sigma_{\mathrm{last}}$ nor $\Omega_{\mathrm{last}}$ is  contained in the exceptional divisor $\Delta_\mathrm{last}$ of the divisorial contraction $\phi\colon M_{\mathrm{last}}\to\overline{M}$.
\end{prop}

\begin{proof}
By Lemma~\ref{lem:EremainsStableIfHom0}, this follows immediately from Lemma~\ref{lem:NoHomsOmega} and Lemma~\ref{lem:NoHomsSigma}.
In fact, let $x=[F]\in\Omega\setminus\Delta$.
By Lemma~\ref{lem:EremainsStableIfHom0}, $x$ is not contained in any center of the flops $f_{c,d}$, and thus $x\in\Omega_{\mathrm{last}}\setminus\Delta_{\mathrm{last}}$, as we wanted.
We argue similarly for $x=[F]\in\Sigma'\setminus\Delta'$.
\end{proof}

\begin{prop}\label{prop:FixedCmptsDoNotIntersectMbar}
The two irreducible components $\overline{\Sigma}, \overline{\Omega} \subset \overline{M}$ do not intersect.
\end{prop}

By assuming Proposition~\ref{prop:FixedCmptsDoNotIntersectMbar}, we can now complete the proof of the Main Theorem.

\begin{proof}[Proof of the Main Theorem, divisibility 2 case.]
Let $\mathcal M \to D$ be the pullback of the universal family over $\Def(\overline M)$ to a general disc $D \subset \Def(\overline M, \overline{\tau})$.
By Proposition \ref{uscire} and Corollary \ref{unaclasse}, $\mathcal M \to D$ is a smoothing of $\overline{M}$ and there is a global involution $\mathfrak{T}$ on $\mathcal M$ which preserves the fibers $\mathcal M_t$ and is such that for $t \neq 0$, $\mathfrak{T}_t$ is the involution whose fixed locus we are interested in and for $t=0$, $\mathfrak{T}_0=\overline \tau$.

By Proposition~\ref{prop:FixedCmptsDoNotIntersectMbar}, the fiber at zero $\Fix(\mathfrak{T},\mathcal M)_0$ of the fixed locus $\Fix(\mathfrak{T},\mathcal M)$ is  the disjoint union $\overline{\Sigma} \coprod \overline{\Omega}$.
Moreover, by Proposition~\ref{prop:FixedCmptsNotContainedInDelta}, both components are $g$-dimensional and at a general point of either component the fixed locus $\Fix(\mathfrak{T},\mathcal M)\to D$ is smooth of relative dimension $g$.
By letting $\mathcal{M}^*:=\mathcal M \setminus \mathcal{M}_0$ and $\mathcal F:=\overline{\Fix(\mathfrak{T},\mathcal{M}^*)}$, we get that any fiber $\mathcal{F}_t$ has exactly two components, for all $t\in D$, thus concluding the proof.
\end{proof}

We are left with the task of proving the two propositions.
The key part of the proof is to understand the linearization of the action of $\tau$ on the line bundle associated to $\lambda$.
Let $G:=\{\mathrm{id},\tau\}\cong\mu_2$ be the group acting on each of the birationals model $M_{c,d}$ and $M_{\mathrm{last}}$ of $M$.
We denote by $\rho_\mathrm{tr}$ and $\rho_\mathrm{det}$ the two irreducible representations of $G$, acting respectively by $(\pm 1)$.
We let $L_{c,d}$, respectively $L_{\mathrm{last}}$ be the line bundle on $M_{c,d}$, respectively $M_{\mathrm{last}}$, with first Chern class $\lambda$.

\begin{theo}\label{thm:linearization}
There is a $G$-linearization on $L_{\mathrm{last}}$ such that
\begin{itemize}
    \item if $x\in \Omega_{\mathrm{last}}$, then $L_{\mathrm{last}\,|x}\cong\rho_{\mathrm{tr}}$;
    \item if $x\in \Sigma_{\mathrm{last}}$, then $L_{\mathrm{last}\,|x}\cong\rho_{\mathrm{det}}$.
\end{itemize}
\end{theo}

The above theorem provides the final ingredient to finish the proof of the Main Theorem in divisibility ~2, which establishes that the fixed locus of $\tau$ consists of two connected components. Theorem \ref{thm:linearization} moreover shows that one can distinguish between these two connected components 
according to the action of $\tau$ on the fibers of $L_{\mathrm{last}}$ at the fixed points. In particular, one of the two components will always be contained in the base locus of the linear system $|\lambda|$. This is exactly what happens in the case of the LLSvS HK manifold associated to cubic fourfolds, where the cubic itself is the base locus of the linear system (\cite{LLSvS:cubics}).

The idea of the proof of Theorem~\ref{thm:linearization} is to first study the linearization on $M$ and then compute how this linearization changes on the various flops $f_{c,d}$.

\begin{prop}\label{prop:LinearizationM}
There is a $G$-linearization on $L=L_{0,-1}$ such that if $x\in \Omega\cup \Sigma$, then $L_{|x}\cong\rho_{\mathrm{tr}}$.
\end{prop}

\begin{proof}
We can write
\[
L=\cO_{M}(\Delta)\otimes\pi^*\cO_{\PP^g}(1).
\]
Each factor is a $G$-invariant line bundle.
We define a $G$-linearization on $L$ by considering the trivial $G$-linearization on $\pi^*\cO_{\PP^g}(1)$ and defining an explicit $G$-linearization on the line bundle $\cO_{M}(\Delta)$ as follows.
We write locally
\[
\mathcal{I}_{\Delta} = (f)
\]
and lift the $G$-action on $\cO_{M}(\Delta)$ by mapping the local generator $1/f$ as follows:
\[
\frac 1f \longmapsto \frac{1}{\tau^*(f)}.
\]

Let $x\in \Omega\cup \Sigma$, i.e., $\tau(x)=x$.
We choose local coordinates
\[
y_1,\dots,y_g, t_1,\dots,t_g
\]
centered at $x$ such that
\[
\begin{cases}
\tau^*(y_i)=-y_i, \qquad &y_i\text{'s coordinates ``on the fibers of }\pi\text{''}\\
\tau^*(t_i)=t_i, \qquad &t_i\text{'s pull-back of coordinates on }\PP^g.
\end{cases}
\]
Since $\tau$ preserves the zero-locus of $f$, we have that
\[
L_{|x} \cong \begin{cases}
\rho_{\mathrm{det}}, & \text{ if }f(-y,t)=-f(y,t),\text{ i.e., if }\mathrm{mult}_{(0,0)}(f)\text{ is odd}\\
\rho_{\mathrm{tr}}, & \text{ if }f(-y,t)=f(y,t),\text{ i.e., if }\mathrm{mult}_{(0,0)}(f)\text{ is even.}
\end{cases}
\]
The statement now follows immediately from Lemma~\ref{lem:MulteplicityFormula} if $x\in\Sigma$, and from Lemma~\ref{lem:NoHomsOmega} if $x\in\Omega$.
\end{proof}

We now return to our study of the first flop: $f\colon M\dashrightarrow M'$ which is the Mukai flop at $\Sigma$.
Then the linearization on $L$ of Proposition~\ref{prop:LinearizationM} on $M$ induces a linearization on $L'=L_{0,0}$ on $M'$.
In fact, it induces a linearization on each $L_{c,d}$ on $M_{c,d}$ and on $L_{\mathrm{last}}$ as well.

\begin{prop}\label{prop:LinearizationMprime}
With respect to the above linearization on $L'$ on $M'$ we have that if $x\in \Sigma'$, then $L'_{ |x}\cong\rho_{\mathrm{det}}$.
\end{prop}

\begin{proof}
Let
\[
m:=\mathrm{deg}(L_{|\gamma}),
\]
where $\gamma\subset\Sigma\cong\PP^g$ is a line. 
Then, in the notation of Diagram \ref{eq:firstflopdiagram} giving the resolution of the Mukai flop $f:=f_{0,-1}$, we have
\[
g^*L \cong g^{\prime*}L'\otimes \cO_{\widetilde{M}}(m\Xi).
\]
In particular, if $x\in\Sigma'$, then
\[
L'_{|x} \cong (-1)^m \rho_{\mathrm{tr}}.
\]
To prove the statement, we need to show that $m$ is odd.

Since $L\cong \cO_M(\Delta)\otimes\pi^*\cO_{\PP^g}(1)$, we have that
\[
m=1+\Delta\cdot\gamma.
\]
By construction, since $\delta=\cl(\Delta$), where $\delta=(2,-h,g/2)$, and the first flop is induced by the class $a_{0,-1}=(1,0,1)$, we get
\begin{equation*}
\Delta\cdot\gamma = \left(\left(2,-h,\frac g2\right),(1,0,1)\right)=-\left( \frac g2 + 2\right).
\end{equation*}
Since $4\,|\,g$, this completes the proof.
\end{proof}

\begin{proof}[Proof of Theorem~\ref{thm:linearization}.]
This follows by a similar argument as for Proposition~\ref{prop:FixedCmptsNotContainedInDelta}.
If $x=[F]\in\Omega\setminus\Delta$, then by Lemma~\ref{lem:EremainsStableIfHom0} together with Lemmas ~\ref{lem:NoHomsOmega} and ~\ref{lem:NoHomsSigma}, we have $L_{\mathrm{last}\,|x}\cong L_{|x}$.  But we know $L_{|x}\cong\rho_{\mathrm{tr}}$, by Proposition~\ref{prop:LinearizationM}.
We argue similarly for $x\in\Sigma'\setminus\Delta'$, by using Proposition~\ref{prop:LinearizationMprime}.
\end{proof}

\begin{proof}[Proof of Proposition~\ref{prop:FixedCmptsDoNotIntersectMbar}.]
The divisorial contraction $\phi\colon M_{\mathrm{last}}\to \overline{M}$ is $G$-equivariant and by definition $L_{\mathrm{last}}\cong\phi^*(\overline{L})$, for a line bundle $\overline{L}$ on $\overline{M}$ together with a $G$-linearization.
Then the statement follows immediately from Theorem~\ref{thm:linearization} and Proposition~\ref{prop:FixedCmptsNotContainedInDelta}.
\end{proof}


\vspace{0.5cm}


\begin{thebibliography}{C-MTB11}

\bibitem[AP06]{AP:tstructures} Abramovich, D., Polishchuk, A., Sheaves of {$t$}-structures and valuative criteria for stable complexes, {\it J.~Reine Angew.~Math.} {\bf 590} (2006), 89--130.

\bibitem[AH-LH18]{AHLH:Moduli} Alper, J., Halpern-Leistner, D., Heinloth, J., Existence of moduli spaces for algebraic stacks, eprint {\tt arXiv:1812.01128v1}.

\bibitem[Apo14]{Apostolov:Moduli} Apostolov, A., Moduli spaces of polarized irreducible symplectic manifolds are not necessarily connected, {\it Ann.~Inst.~Fourier (Grenoble)} {\bf 64} (2014), 189--202.

\bibitem[ACGH85]{acgh} Arbarello, E., Cornalba, M., Griffiths, P., Harris, J., Geometry of algebraic curves, Vol. I, {\it Grundlehren der Mathematischen Wissenschaften} {\bf 267}, Springer-Verlag, New York, 1985.

\bibitem[AB13]{ArcaraBertram:BridgelandKtrivial} Arcara, D., Bertram, A., Bridgeland-stable moduli spaces for K-trivial surfaces, with an appendix by M. Lieblich, {\it J.~ Eur.~ Math.~ Soc.~ (JEMS)} {\bf 15} (2013), 1--38.

\bibitem[ASF15]{ASF:Prym} Arbarello, E., Sacc\`a, G., Ferretti, A., Relative Prym varieties associated to the double cover of an Enriques surface, {\it J.~ Diff.~ Geom.} {\bf 100} (2015), 191--250.

\bibitem[Bay18]{Bayer:BN} Bayer, A., Wall-Crossing implies Brill-Noether. Applications of stability conditions on surfaces, in {\it Algebraic geometry: Salt Lake City 2015}, Proc.~Sympos.~Pure Math.~ {\bf 97.1} (2018), 3--27.

\bibitem[BB17]{BayerBridgeland:StabK3} Bayer, A., Bridgeland, T., Derived automorphism groups of K3 surfaces of Picard rank 1, {\it Duke Math.\ J.} {\bf 166} (2017), 75--124.


\bibitem[BL+21]{BLMNPS:family} Bayer, A., Lahoz, M., Macr\`i, E., Nuer, H., Perry, A., Stellari, P., Stability conditions in family, {\it Publ. Math. IHES} {\bf 133} (2021),  157--325.


\bibitem[BaMa14a]{BM:projectivity} Bayer, A., Macr\`i, E., Projectivity and birational geometry of Bridgeland moduli spaces, {\it J. Amer. Math. Soc.}  {\bf 27} (2014), 707--752.

\bibitem[BaMa14b]{BM:walls} \bysame, MMP for moduli of sheaves on K3s via wall-crossing: nef and movable cones, Lagrangian fibrations, {\it Invent.~Math.}  {\bf 198} (2014),  505--590.

\bibitem[Bea83]{Bea:MukaiInvolution} Beauville A., Some remarks on K\"ahler manifolds with $c_1=0$, in \emph{Classification of algebraic and analytic manifolds (Katata, 1982)}, Progr.~Math.~ {\bf 39} (1983), 1--26.

\bibitem[Bea86]{beaumon} \bysame, Le groupe de monodromie des familles universelles d'hypersurfaces et d'intersections compl\`etes, in \emph{Complex analysis and algebraic geometry (G\"ottingen, 1985)}, Springer LNM.~ {\bf 1194} (1986), 8--18.

\bibitem[Bea99]{Beauville-counting} \bysame, Counting rational curves on {$K3$} surfaces, {\it Duke Math.~J.} \textbf{97} (1999), 99--108.

\bibitem[Bea11]{beauville:involutions} \bysame, Antisymplectic involutions of holomorphic symplectic manifolds, {\it J.~Topol.} {\bf 4} (2011), 300--304.

\bibitem[BD82]{BD:cubic} Beauville, A., Donagi, R., La vari\'et\'e des droites d'une hypersurface cubique de dimension 4, {it C. R. Acad. Science Paris} {\bf 301} (1982), 703--706.



\bibitem[BFM21]{BFM:DV} Bernardara, M., Fatighenti, E., Manivel, L., Nested varieties of K3 type, {\it J. \'Ec. polytech. Math.} {\bf 8} (2021) 733--778.


\bibitem[Bri07]{Bridgeland:Stab} Bridgeland, T., Stability conditions on triangulated categories. {\it Ann. of Math. (2)} {\bf 166} (2007), 317--345.

\bibitem[Bri08]{Bridgeland:StabK3} \bysame, Stability conditions on K3 surfaces, {\it Duke Math.\ J.} {\bf 141} (2008), 241--291.

\bibitem[CCL18]{CCL:ChowLLSvS} Camere, C., Cattaneo, A., Laterveer, R., On the Chow ring of certain Lehn-Lehn-Sorger-van Straten eightfolds, to appear in Glascow Math. J.

\bibitem[C-MF05]{CMF:CubicThreefolds} Casalaina-Martin, S., Friedman, R., Cubic threefolds and abelian varieties of dimension five,
{\it J.\ Algebraic Geom.} {\bf 14} (2005), 295--326.

\bibitem[C-MTB11]{CMTB:Singularities} Casalaina-Martin, S., Teixidor i Bigas, M., Singularities of Brill-Noether loci for vector bundles on a curve, {\it Math.\ Nachr.} {\bf 284} (2011), 1846--1871.

\bibitem[Coo93]{Cook-thesis} Cook, P., Local and Global Aspects of the Module Theory of Singular Curves, Ph.D.~Thesis, University of Liverpool, 1993.

\bibitem[Deb18]{Debarre:Survey} Debarre, O., Hyperk\"ahler manifolds, eprint {\tt arXiv:1810.02087}.

\bibitem[DV10]{DV:HK} Debarre, O., Voisin, C., Hyper-K\"ahler fourfolds and Grassmann geometry, {\it J.~Reine Angew.~Math.} {\bf 649} (2010), 63--87.

\bibitem[dCRS19]{dCRS} de Cataldo, M., Rapagnetta, A., Sacc\`a, G., The Hodge numbers of O'Grady 10 via Ng\^o strings, to appear in Jour. Math. Pur. Appl .


\bibitem[FV14]{Farkas-Verra} Farkas, G., Verra, A., The geometry of the moduli space of odd spin curves, {\it Ann.~of Math.~(2)} {\bf 180} (2014), 927--970.

\bibitem[FM21]{FM:FanoK3} Fatighenti, E., Mongardi, G., Fano varieties of K3 type and IHS manifolds, {\it Int. Math. Res. Not. IMRN} {\bf 4} (2021) 3097--3142

\bibitem[Fer11]{Ferretti:thesis} Ferretti, A., The Chow ring of double EPW sextics, {\it Rend. Mat. Appl.} {\bf 31} (2011), 69--217.




\bibitem[Har82]{Harris-theta} Harris, J., Theta-characteristics on algebraic curves, {\it Trans.~Amer.~Math.~Soc.} {\bf 271} (1982), 611--638.

\bibitem[HLS19]{Hulek-Laza-Sacca} Hulek, K., Laza, R., Sacc\`a, G., The Euler number of hyper-K\"ahler manifolds of OG10 type, Mat.~ Contemp.~ {\bf 47} (2020), 151--170.

\bibitem[Huy97]{Huy:bir} Huybrechts, D., Birational symplectic manifolds and their deformations, {\it J. Diff. Geom.} {\bf 45} (1997), 488--513.


\bibitem[Huy12]{Huybrechts:Torelli} \bysame, A global Torelli theorem for hyperk\"ahler manifolds [after M. Verbitsky], S\'eminaire Bourbaki: Vol. 2010/2011, expos\'ees 1027--1042, {\it Ast\'erisque} {\bf  348} (2012), 375--403.

\bibitem[IM11]{IM:doubleEPW} Iliev, A., Manivel, L., Fano manifolds of degree ten and EPW sextics, {\it Ann.~Sci.~\'Ec.~Norm.~Sup\'er.} {\bf 44} (2011), 393--426.

\bibitem[IM15]{IM:FanoCY} \bysame, Fano manifolds of Calabi-Yau Hodge type, {\it  J.~Pure Appl.~Algebra} {\bf 219} (2015), 2225--2244.

\bibitem[Jan83]{janssen1} Janssen, W., Skew-symmetric vanishing lattices and their monodromy groups, \emph{Math.~Ann.} {\bf 266}, (1983), 115--133.

\bibitem[Jan85]{janssen2} \bysame, Skew-symmetric vanishing lattices and their monodromy groups. II, \emph{Math.~Ann.} {\bf 272}, (1985), 17--22.

\bibitem[KS08]{KS:stability} Kontsevich, M., Soibelman, Y., Stability structures, motivic Donaldson-Thomas invariants and cluster transformations, eprint {\tt arXiv:0811.2435}.



\bibitem[Kuz16]{Kuz:Kuchle} Kuznetsov, A., K\"uchle fivefolds of type c5, {\it Math.~Z.} {\bf 284} (2016), 1245--1278.

\bibitem[Leh15]{Lehn:Oberwolfach} Lehn, M., Twisted cubics on a cubic fourfold and in involution on the associated 8-dimensional symplectic manifold, in {\it Oberwolfach Report} No. 51/2015, 22--24, 2015.

\bibitem[Leh18]{Lehn:singcub} Lehn, C., Twisted cubics on singular cubic fourfolds - On Starr's fibration, {\it Math. Z.} {\bf 290} (2018), 379--388.

\bibitem[LLSvS17]{LLSvS:cubics} Lehn, C., Lehn, M., Sorger, C., van Straten, D., Twisted cubics on cubic fourfolds, {\it J. Reine Angew. Math.} {\bf 731} (2017), 87--128.


\bibitem[LP93]{LePotier:Fitting} Le Potier, J., Faisceaux semi-stables de dimension 1 sur le plan projectif, {\it Rev.~Roumaine Math.~Pures Appl.} {\bf 38} (1993), 635--678.

\bibitem[LPZ18]{LPZ:cubics} Li, C., Pertusi, L., Zhao, X., Twisted cubics on cubic fourfolds and stability conditions, eprint {\tt arXiv:1802.01134}.

\bibitem[Lie06]{Lieblich:moduli} Lieblich, M., Moduli of complexes on a proper morphism, {\it J.\ Algebraic Geom.} {\bf 15} (2006), 175--206.

\bibitem[Lin20]{Lin} Lin, H.-Y., Lagrangian constant cycle subvarieties in Lagrangian fibrations, {\it Int. Math. Res. Not. IMRN} {\bf 1} (2020), 14--24.

\bibitem[MS21]{MS:Mumford} Macr\`i, E., Schmidt, B., Stability and applications, {\it Pure Appl. Math. Q.} {\bf 17} (2021), 671--702.


\bibitem[Mar01]{Mar:BN} Markman, E., Brill-Noether duality for moduli spaces of sheaves on K3 surfaces, {\it J.~Algebraic Geom.} {\bf 10} (2001), 623--694.

\bibitem[Mar10a]{markman:modular} \bysame, Modular Galois covers associated to symplectic resolutions of singularities, {\it J.~Reine Angew.~Math.} {\bf 644} (2010), 189--220.

\bibitem[Mar10b]{Markman:Monodromy} \bysame, Integral constraints on the monodromy group of the hyperk\"ahler resolution of a symmetric product of a K3 surface, {\it Internat.~J.~of Math.} {\bf 21} (2010), 169--223.

\bibitem[Mar11]{Markman:Survey} \bysame, A survey of Torelli and monodromy results for holomorphic-symplectic varieties, in {\it Complex and differential geometry}, 257--322, Springer Proc.~Math.~{\bf 8}, Springer, Heidelberg, 2011.

\bibitem[Mar13]{Markman:PrimeExceptional} \bysame, Prime exceptional divisors on holomorphic symplectic varieties and monodromy reflections, {\it Kyoto J.~Math.} {\bf 53} (2013), 345--403.

\bibitem[Muk84]{Muk:Sympl} Mukai, S., Symplectic structure of the moduli space of sheaves on an abelian or K3 surface, {\it Invent. Math.} {\bf 77} (1984), 101--116.

\bibitem[Muk87]{Muk:K3} \bysame, On the moduli space of bundles on K3 surfaces. I, in {\it Vector bundles on algebraic varieties} (Bombay, 1984), 341--413, Tata Inst. Fund. Res. Stud. Math. {\bf 11}, Tata Inst. Fund. Res., Bombay, 1987.

\bibitem[Mum66]{mum-eq-abvars}
Mumford D., On the equations defining abelian varieties. {I}.,
 {\it Invent.~Math.} {\bf 1} (1966), 287--354.

\bibitem[Mum68]{Mumford:Chow} \bysame, Rational equivalence of 0-cycles on surfaces, {\it J.~Math.~Kyoto Univ.} {\bf 9} (1968), 195--204.

\bibitem[Mum71]{mumtheta} \bysame, Theta characteristics of an algebraic curve, {\it  Ann.~Sci.~\'Ecole Norm.~Sup.} {\bf 4} (1971), 181--192.


\bibitem[Nam01]{Namikawa:Def} Namikawa, Y., Deformation theory of singular symplectic n-folds, {\it Math.~Ann.} {\bf 319} (2001), 597--623.

\bibitem[O'G97]{Kieran:weight2} O'Grady, K., The weight-two Hodge structure of moduli spaces of sheaves on a $K3$ surface, {\it J. Algebraic Geom.}, {\bf 6} (1997), 599--644.

\bibitem[O'G05]{Kieran:Involutions} \bysame, Involutions and linear systems on holomorphic symplectic manifolds, {\it GAFA, Geom.~funct.~anal.} {\bf 15} (2005), 1223--1274.

\bibitem[O'G06]{Kieran:doubleEPW} \bysame, Irreducible symplectic 4-folds and Eisenbud-Popescu-Walter sextics, {\it Duke Math. J.} {\bf 134} (2006), 99--137.


\bibitem[Ray82]{Raynaud} Raynaud, M., Sections des fibr\'es vectoriels sur une courbe. {\it Bull.~Soc.~Math.~France} \textbf{110} (1982), 103--125. 

\bibitem[Reg80]{Rego} Rego, C., The compactified {J}acobian, {\it Ann.~Sci.~\'{E}cole Norm.~Sup.~(4)} \textbf{13} (1980), 211--223.

\bibitem[Roi80]{Roi:0cycles} Roitman, A., Rational equivalence of zero-dimensional cycles, {\it Mat.~Zametki}
{\bf 28} (1980), 85--90, 169.



\bibitem[Tod08]{Toda:ModuliK3} Toda, Y., Moduli stacks and invariants of semistable objects on K3 surfaces, {\it Adv. Math.} {\bf 217} (2008), 2736--2781.


\bibitem[Ver13]{Verbitsky:torelli} Verbitsky, M., Mapping class group and a global {T}orelli theorem for hyperk\"ahler manifolds, with an appendix by E. Markman, {\it {D}uke {M}ath. {J}.} {\bf 162} (2013), 2929--2986.

\bibitem[Voi03]{Voisin:book2} Voisin, C., Hodge theory and complex algebraic geometry II,
{\it Cambridge Studies in Advanced Mathematics}, Cambridge University Press, Cambridge, 2003.

\bibitem[Wel81]{Welters:thesis} Welters, G., Abel-Jacobi isogenies for certain types of Fano threefolds, {\it Mathematical Centre Tracts} {\bf 141}, Mathematisch Centrum, Amsterdam, 1981.

\bibitem[Yos99]{Yos:Mukai} Yoshioka, K., Some examples of Mukai's reflections on K3 surfaces, {\it J.~Reine Angew.~Math.} {\bf 515} (1999), 97--123.

\bibitem[Yos01a]{Yos:BN} \bysame, Brill-Noether problem for sheaves on K3 surfaces, in {\it Proceedings of the Workshop ``Algebraic Geometry and Integrable Systems related to String Theory'' (Kyoto, 2000)} {\bf 1232} (2001), 109--124.

\bibitem[Yos01b]{Yos:moduli} \bysame, Moduli spaces of stable sheaves on abelian surfaces, {\it Math. Ann.} {\bf 321} (2001), 817--884.

\end{thebibliography}
\end{document}